\documentclass[preprint]{elsarticle}

\usepackage[OT1]{fontenc}
\usepackage{amsthm, amssymb, amsmath, amsfonts}
\usepackage[colorlinks]{hyperref}
\usepackage{graphicx}
\usepackage{epstopdf}
\usepackage[usenames,dvipsnames]{color}
\usepackage{comment}

\begin{document}
\font\eufm = eufm10 scaled 1200
\def\u{{\bf u}}
\def\Beta{{\cal B}}
\def\cov{\mbox{cov}}
\def\corr{\mbox{corr}}
\def\D{{\mathbb D}}
\def\E{{\mathbb E}}
\def\Ff{\mbox{\eufm F}}
\def\N{{\cal N}}
\def\Nb{\mathbb N}
\def\P{{\mathbb P}}
\def\Rb{\mathbb R}
\def\vv{{\bf v}}
\def\X{{\bf X}}
\def\x{{\bf x}}
\def\di{\displaystyle}
\def\indi{1\hspace{-1,1mm}{\rm I}}
\def\toi#1{\mbox{\kern+5pt \hbox to 25pt{\rightarrowfill}
 \kern-28pt \raise-5pt   \hbox{$\scriptstyle #1 \to \infty $ }\kern+1pt }}
\def\too#1{\mbox{\kern+3pt \hbox to 23pt{\rightarrowfill}
 \kern-23pt \raise-5pt   \hbox{$\scriptstyle #1 \to 0 $} \kern+0pt }}
\def\tow{\stackrel{W}{\rightarrow}}
\def\top{{\;\buildrel \bf P  \over{\toi{n}}\;}}
\def\tr{{\rm \bf tr}}
\def\Var{{\mathbb V}\!\mbox{ar}}
\def\Vol{\mbox{Vol}}
\def\essinf{\mbox{\rm ess\hspace{+1pt}inf}}

\def\red{\color{red}}
\def\epsilon{\varepsilon}

\newtheorem{theorem}{Theorem}[section]
\newtheorem{lemma}{Lemma}[section]
\newtheorem{proposition}{Proposition}[section]
\newtheorem{corollary}{Corollary}[section]
\newtheorem{remark}{Remark}[section]

\begin{frontmatter}

\title{Coverage Problem revisited}

\author[spbu]{Alexey~Antonik\fnref{fn1}}
\ead{alexey.antonik@gmail.com}
\author[havu]{Alexandre Berred}
\ead{alexandre.berred@univ-lehavre.fr}
\author[spbu,spbeu]{Sergey~V.~Malov\fnref{fn1}\corref{cor}}
\ead{malovs@sm14820.spb.edu}

\cortext[cor]{Corresponding author}
\fntext[fn1]{These authors were supported in part by Russian Ministry of Education and Science Mega-grant no.11.G34.31.0068 (Dr. Stephen J. O'Brien, Principal Investigator).}

\address[spbu]{Theodosius Dobzhansky Center for Genome Bionformatics, St.-Petersburg State University, 
 199034, Sredniy avenue 41A, St.-Petersburg, Russia}
\address[havu]{Universit\'e du Havre, UFR Sciences et Techniques, BP 540, 76058, Le Havre
Cedex, France}
\address[spbeu]{Dept. of mathematics, St.-Petersburg Electrotechnical University ``LETI'',\\  197376, Prof. Popova str. 5, St.-Petersburg, Russia}

\begin{abstract}
Motivated by some problems in genome assembling, we investigate
properties of spacings from absolutely continuous distributions. Several results
on the asymptotic behavior of the maximal uniform and non-uniform $k$-spacings are
presented. Applications of these results to the coverage problem for the prediction
stage in genome assembling are also provided.
\end{abstract}

\begin{keyword}
coverage problem \sep spacings \sep uniform spacings \sep extremal types

\MSC primary 60F05 \sep secondary 60G70 \sep 62G30 \sep 62G32 \sep 92D20
\end{keyword}

\end{frontmatter}

\section{Introduction}
\label{sec:intro}

With the development of next generation random shotgun sequencing technology 
it is important to construct accurate estimations {\it a priori} of the fold
coverage, required for a 
confident sampling of the sequenced genome.    

The human genome contains approximately $3.2\cdot 10^9$ pairs of bases (bp). 
Next generation sequencers (NGS) process multiple copies of the original DNA or its
parts. However, due to the technological restrictions only short reads (about 200 bp) of these copies can be acquired. 
Thus, in order to reconstruct the entire genome one needs to assemble acquired reads. The latter is possible only 
if the source genome is covered contiguously with overlaps of the lengths prescribed in advance. 

The question posed in this paper is: {\it ``How many reads we need to cover the 
whole genome
with sufficiently large probability?''}. It is important to mention that some 
reads may have 
errors. 
Therefore, with an eye to resolve these errors one is required to cover the entire genome several times ($r$-times).
 
Modern assembling techniques assume that the cumulative length of reads exceeds
40 times the length of original DNA sequence. 
Lander and Waterman \cite{lwr88} introduced an efficient heuristic method 
for evaluating the genome coverage 
taking into account possible gaps and targets in the genome. 
Roach \cite{roa95} improved these results by introducing spacings and performing considerations in the increasing scale, {\em i.e.} when $n\to\infty$. 
By using properties of the uniform spacings, similar results on coverage of the interval $[0,1]$ were obtained, among others,
in \cite{lev39, wss59, sig78, sig79, shl82, hul03}.

We are going to study the problem of full $r$-times coverage of the interval $[0,1]$
by random subsegments of small length under uniform and other distributions of the subsegments locations.
The entire genome is represented by the interval $[0,1]$, reads are subsegments of length $Y_i=L_i/N$, where $L_i$ is
the length of $i$-th read and $N$ is the total length of the genome sequence.
The locations of reads, determined by its left ends, are assumed to be randomly distributed over the entire genome.
The length of reads $L_i$ can be either random or deterministic. First, we
assume that it is deterministic and the reads are of same length,
{\em i.e.} $L_i=L$ and $Y_i=L/N$ for any $i$. 
In this case, the $i$-th read is characterized only by its left end, which is a random variable $X_i$ varying in
$[0,1]$. In addition, we assume that $X_1$,~\ldots, $X_n$ are independent and
identically distributed ({\em i.i.d.}) random variables, where $n$ is the number of available reads.

The problem of covering the interval $[0,1]$ by random segments of
fixed length $l$, $0<l<1$, can be rewritten in terms of the maximal uniform
spacings. More precisely, denote by $X_{1,n}\leq\ldots\leq X_{n,n}$ the order
statistics associated to $X_1,\ldots,X_n$ and define the spacings
$S_{0,n}=X_{1,n}$, $S_{i,n}=X_{i+1,n}-X_{i,n}$, $i=1,\ldots,n-1$ and
$S_{n,n}=1-X_{n,n}$. Denote by $M_{n+1,n}\leq\ldots\leq
M_{1,n}$ the order statistics corresponding to $S_{0,n}$,~\ldots, $S_{n,n}$.
Thus the coverage probability of the interval $[0,1]$ by segments of length $l$ is equal to $\P(M_{1,n}<l)$.

To the best of our knowledge, exact formula for this probability was obtained in \cite{lev39} (see also \cite{km63}). 
However, this result is absolutely unpractical for large $n$. Asymptotic distribution of maximal and minimal uniform spacings were obtained in \cite{wss59}. 
Asymptotic behaviour of maximal $r$-spacings were investigated by Holst \cite{hst80}. Dembo \& Karlin \citep{dkr92} studied asymptotic formulas and its rate of convergence for the $k$-th maximal and $k$-th minimal uniform $r$-spacings as a particular case of the $r$-scan processes by using Chen--Stein method of Poisson approximations (see \cite{chn75}). 
They had also extended asymptotic formulas to the some cases of non-uniform $r$-spacings. 
Some further investigations were performed, among others, in \citep{glz94}, where a number of such asymptotics were studied and compared to each other. 
Note that asymptotics for minimal $r$-spacings was obtained for a wide class of probability density functions (PDFs) of the initial variables. 
At the same time, for the maximal $r$-spacings such results were obtained under the piecewise-constant PDFs only. 
Barbe \cite{brb92} obtained independently an approximation formula for maximal 
spacings from  bounded support non-uniform distributions with PDF 
bounded away from zero.
Deheuvels \cite{dvs86} derived general results for maximal spacings under the wide class of absolutely continuous distributions.
The coverage problem for this class was further developed in \cite{hus88}.

The problem of $r$-times coverage of the interval $[0,1]$ by random segments of length $l$ can be reformulated in terms of $r$-spacings
$S_{0,n}^{(r)}=X_{r,n}$,
$S_{i,n}^{(r)}=X_{i+r,n}-X_{i,n}$ for $i=1,\ldots,n-r$, and
$S_{n-r+1,n}^{(r)}=1-X_{n-r+1,n}$.
Denote by $M_{n-r+2,n}^{(r)}\leq\ldots\leq M_{1,n}^{(r)}$ order statistics associated to them.
The probability of $r$-times coverage of the interval $[0,1]$
by random segments of length $l$ is equal to $\P(M_{1,n}^{(r)}<l)$. 
Of course, the spacings and the $1$-spacings are the same objects,  
therefore we will keep the notation for them throughout the paper.

The paper is organized as follows. The uniform case is studied in 
Section~\ref{sec:unif}, where we present some improved asymptotic results for 
the uniform $r$-spacings.
Some ideas on the exact distribution of the uniform $r$-spacings are expounded 
in Subsection~\ref{sub:eunif}.
Asymptotic results on the distribution of maximal $r$-spacings for special 
cases of non-uniform bounded support distributions are
presented in Section~\ref{sec:mixed}.
Extension of the asymptotic results due to \cite{dvs86} to the case of $r$-spacings are given in Section~\ref{sec:wide}.
Simulation and numerical results are provided in Section~\ref{sec:simul}.
Assessment of the weak convergent are given in Subsection~\ref{sub:simul}.
Applications to the genome coverage are considered in Subsection~\ref{sub:covapp}.
Proofs are postponed until Section~\ref{sec:proofs}.

\section{Uniform distribution}
\label{sec:unif}
\subsection{Approximation results}\label{sub:unif}

Let $U_1,\ldots,U_n$ be {\em i.i.d.} standard uniform random variables. As
in the introduction, denote by
$U_{1,n}\leq\ldots\leq U_{n,n}$
the associated order statistics and set, by convention, $U_{0,n}=0$, and $U_{n+1,n}=1$.
Consider the corresponding $r$-spacings
$$S_{i,n}^{(r)}=U_{i+r,n}-U_{i,n},\quad i=1,\ldots,n-r,$$
$S_{0,n}^{(r)}=U_{r,n}$ and $S^{(r)}_{n-r+1,n}=1-U_{n-r+1,n}$,
and the corresponding $r$-spacings order statistics
$M_{n-r+2,n}^{(r)}\leq\ldots\leq M_{1,n}^{(r)}$.
Let $E_1,E_2,\ldots$ be a sequence of {\em i.i.d.} random variables having
the exponential distribution with the rate $1$. Denote by $S_n=\sum_{i=1}^n
E_i$ the partial sums associated to $E_1,E_2,\ldots$ ($S_0=0$).

As is well known (see for instance \cite{pyk65}), the vector of uniform order statistics
$\{U_{i,n},\;0\leq i\leq n+1\}$ is equal in distribution
to $\{S_i/S_{n+1},\;0\leq i\leq n+1\}$. Then
the corresponding $r$-spacings $\{S_{i,n}^{(r)},\;0\leq i\leq  n-r\}$ are
equal in distribution to $\{(S_{i+r}-S_i)/S_{n+1},\;0\leq i\leq  n-r\}$.
Therefore the distribution of the $k$-th
maximal $r$-spacing has the same distribution as the $k$-th maximum of
$\{(S_{i+r}-S_i)/S_{n+1},\;0\leq i\leq  n-r\}$. In particular for $k=1$,
$$M_{1,n}^{(r)}\buildrel {\mbox{\scriptsize d}} \over=\max
\{(S_{i+r}-S_i)/S_{n+1},\;0\leq i\leq  n-r\}.$$
By Theorem~7.4 of \cite{rot86}, the
limit behavior in distribution of
$\widetilde M_n^{(r)}=\max\{S_{i+r}-S_i,\;0\leq i\leq  n-r\}$
and  $\widehat M_n^{(r)}=\max\{\Gamma_1,\ldots,\Gamma_n\}$ are the same, where
$\Gamma_1,\ldots,\Gamma_n$ are {\em i.i.d.} random
variables having the $\Gamma(r,1)$ distribution, and
\begin{equation}
\label{equ:limmaver}
\lim_{n\to\infty}\P(\widetilde M_n^{(r)} - b_n^{(r)} \leq x)=
\lim_{n\to\infty}\P(\widehat M_n^{(r)} - b_n^{(r)} \leq x) = e^{{-e}^{-x}},
\end{equation}
where $b_n^{(r)} = \log n + (r - 1) \log \log n - \log \Gamma(r)$. It follows
that
$$
\lim_{n\to\infty}\P(S_{n+1}M_{1,n}^{(r)} - b_n^{(r)} \leq x) = e^{{-e}^{-x}},
$$
Since, by the classical law of large numbers, $S_{n+1}/n \to 1$ in probability as $n\to\infty$,
we obtain
$$\lim_{n\to\infty}\P(n M_{1,n}^{(r)} - b_n^{(r)} \leq x) = e^{{-e}^{-x}}.$$
This asymptotic result was given by Holst \cite{hst80}. The same formula was obtained in \cite{dkr92} via the Poisson approximation. Unfortunately for $r>1$ it does not give a good approximation even for a such large $n$ as $10^7$.

Another approach provides the Proposition \ref{pro:inda}. Indeed,
$$
\P(n M_{1,n}^{(r)}<u)\sim \P(\widetilde M_n^{(r)}<u)\geq \P(\widehat M_n^{(r)}<u)
$$
for all $u$. By using Theorem~1.5.1 of \cite{llr83}, we have
$$
\P(\widehat M_n^{(r)}<u_n)=(1-(1-G_r(u_n)))^n\sim e^{-n(1-G_r(u_n))}
$$
under the assumption that $n(1-G_r(u_n))\to e^{-x}$ as $n\to\infty$, where
$
G_r(x)=\int_0^{x} \frac{t^{r-1}}{\Gamma(r)} e^{-t}\,dt =1-
\sum_{i=1}^r \frac{x^{i-1}}{\Gamma(i)} e^{-x}
$
is the cumulative distribution function (CDF) of $\Gamma(r,1)$.
It turns out that
\begin{equation}
\label{equ:ldappr}
\P(\widehat M_n^{(r)}<u)\sim e^{-n(1-G_r(u))}
\end{equation}
gives a good approximation of $\P(\widehat M_k^{(r)}<u)$. We should remark that the Poisson approximation in \cite{dkr92} leads to the same estimator for $\P(M_{1,n}^{(r)}<u)$, which we have constructed by considering $n-r$ independent random variables of the $\Gamma(r,1)$-distribution.

In order to improve the approximation (\ref{equ:limmaver}), we need to replace the sequence $b_{n}^{(r)}$
with $b_{n,x}^{(r)}$ depending on $x$ and satisfying the equality
$$e^{b_{n,x}^{(r)}}/n= \sum\nolimits_{i=1}^r (b_{n,x}^{(r)}+x)^{i-1}/\Gamma(i)$$
or, equivalently,
\begin{equation}
\label{equ:sfteq}
b_{n,x}^{(r)}-\log n - \log\Bigl(\sum\nolimits_{i=1}^r
(b_{n,x}^{(r)}+x)^{i-1}/\Gamma(i)\Bigr)=0.
\end{equation}
For instance, the best estimate based on (\ref{equ:limmaver}) for the $95\%$
quantile of $\widehat M_k^{(r)}$ is obtained from (\ref{equ:sfteq}) for $x=-\log(-\log(0.95))\approx 2.97$.

\subsection{Exact results}\label{sub:eunif}

Denote by $\lambda_n$ the Lebesgue measure on $\mathbb{R}^n$.
Note that $(S_{0,n},\ldots,S_{n,n})\in\Delta^{n}(1)$,
where
$$
\Delta^{n}(p)=\Bigl\{y = (y_1, \ldots, y_{n+1})\in\mathbb{R}^{n+1}
\;\Big|\; y_1 \geq 0, \ldots, y_{n+1} \geq 0, \sum_{i=1}^{n+1} y_i = p\Bigr\}
$$
is an $n$-dimensional simplex of $\mathbb{R}^n$. The distribution function of
the maximal $r$-spacing can be expressed as
$$
\P(M^{(r)}_{1,\,n} \leq a) = \frac{\lambda_{n}\left(\bigl\{x\in\Delta^{n}(1)
\;\big|\;
\sum_{i=1}^{r} x_i \leq a, \ldots, \sum_{i=n-r+2}^{n+1} x_i \leq
a\bigr\}\right)}{\lambda_{n}\left(\Delta^{n}(1)\right)},
$$
$0\leq a\leq 1$. Note that
$$
{\Bigl\{x\in\Delta^{n}(1) \;\Big|\; \sum_{i=1}^{r} x_i \leq a, \ldots,
\sum_{i=n-r+2}^{n+1} x_i \leq a\Bigr\}} = \Delta^{n}(1) \setminus
\bigcup_{i=1}^{n-r+2} A_i,
$$
where $A_i=\bigl\{x \in\Delta^{n}(1)
\;\big|\; \sum_{j=i}^{i+r-1}x_j \geq a \bigr\}$. We have by the inclusion-exclusion
principle
$$
\lambda_{n}\left(\bigcup_{i=1}^{n-r+2} A_i\right) = \sum_{i=1}^{n-r+2} (-1)^{i -
1}
\mspace{-10mu}
\sum_{\substack{\scriptscriptstyle
J\subseteq\{1,\,\ldots,\,n-r+2\}\\\scriptstyle\#J=i}}
\mspace{-10mu}
\lambda_{n}\left(\bigcap_{j\in J}A_j\right).
$$

It can be shown that for any set $J\subseteq\{1,\,\ldots,\,n-r+2\}$
with $k$ elements, there exists $u\in\mathbb{R}^{n+1}$ such
that
the following identity holds
$$
\bigcap_{j\in J}A_j = \Delta^{n}(\max\{0,\,1-ka\}) + \sum_{j\in J}
\Delta^{r-1}_j(a) + \{u\},
$$
where symbols $+$ and $\sum$ shall be understood in the Minkowski sense and
\arraycolsep=2pt
\begin{eqnarray*}
\Delta^{r-1}_j(a)= \Bigl\{y\in\mathbb{R}^{n+1}\; \Big|\; y_1&=& 0, \ldots, y_{j-1} = 0,\\
y_j \geq 0, \ldots, y_{j+r-1} &\geq& 0, y_{j+r} = 0, \ldots, y_{n+1} = 0,
\sum_{i=j}^{j+r-1} y_i = a\Bigr\}
\end{eqnarray*}
is a projection of $\Delta^{n}(a)$ onto some $r$-dimensional subspace of $\mathbb{R}^{n+1}$.
We finally obtain
$$
\mspace{-275mu}\P(M^{(r)}_{1,\,n} \leq a) = 1 - \lambda_{n}^{-1}\left( \Delta^{n}(1)\right)\times {}
$$
$$\mspace{25mu}{} \times\displaystyle\sum_{i=1}^{n-r+2} (-1)^{i -
1} \mspace{-15mu}\sum_{\substack{\scriptscriptstyle
J\subseteq\{1,\,\ldots,\,n-r+2\}\\\scriptstyle\#J=i}} \mspace{-15mu}
\lambda_{n}\Bigl(
\Delta^{n}(\max\{0,\,1-ia\}) + \sum_{j\in J} \Delta^{r-1}_j(a)
 \Bigr)
  \,.
$$

Without going into details on how the Lebesque measure of a Minkow\-ski sum of 
simplexes can be computed,
we consider only the case $r=1$. It easily follows from the above result that
$$
\P(M^{(1)}_{1,\,n} \leq a) = 1 - \sum\nolimits_{j=1}^{n+1} (-1)^{j - 1}
\binom{n + 1}{j}\max\{0,\,1-ja\}^{n}.
$$

\section{Bounded support distributions}
\label{sec:mixed}

In this section, we consider the limit behavior of the maximal $r$-spacings 
based on 
absolutely continuous distributions having bounded support $[A,B]$.
Throughout the section, we keep the notation introduced in 
Section~\ref{sec:intro}, except that we redefine
$S_{0,n}^{(r)}=X_{r,n}-A$ and $S_{n-r+1,n}=B-X_{n-r+1,n}$.
Our objective is to extend the results of uniform
$r$-spacings. For fixed $r\ge 1$, let $(f_{n,r})_{n\in\Nb }$ be a sequence of
uniformly equicontinuous and non-negative functions such that
\begin{equation}
\label{equ:apprf}
\sup_{x>0}|\P(n M_{1,n}^{(r)}<x)-\exp(-n f_{n,r}(x))|\to 0\quad\mbox{as}\quad n\to\infty,
\end{equation}
where $M_{1,n}^{(r)}$ is the uniform maximal $r$-spacing. Set
$f_r(x)=1-G_r(x)$. Recall from Section~\ref{sec:unif} that $G_r(x)$ is the
CDF of $\Gamma(r,1)$. It follows that
\begin{equation}
\label{equ:mix2}
n(f_{n,r}(x_n)-f_{r}(x_n))\to 0\quad\mbox{as}\quad n\to\infty
\end{equation}
uniformly for all sequences $(x_n)$ and $n f_{n,r}(x_n)\to c$ for any sequence $(x_n)$ such that $n f_r(x_n)\to c$ as $n\to\infty$ for $c\in (\delta,M)$, where $\delta$ and $M$
are any fixed positive constants, $0<\delta<M<\infty$.

First, we state the following result for step PDFs, which can be construed 
as a result for
mixing uniform probability densities. 

\begin{theorem}
\label{teo:mix3}
Let $X_1$,~\ldots, $X_n$ be a sample from an absolutely continuous distribution
with the support $[A,B]$, $-\infty<A<B<\infty$, having PDF
$$
p(x)=
\begin{cases}
c_i \quad\mbox{for}\quad x\in [x_i,x_{i+1}), \; i=0,\ldots,m-1,\cr
c_m \quad\mbox{for}\quad x\in [x_m,x_{m+1}],\cr
0 \quad\mbox{otherwise},\cr
\end{cases}
$$
for some $A=x_0<\ldots<x_{m+1}=B$ and $c_i>0$, $i=0,\ldots,m$. Then,
\begin{equation}
\label{equ:mix3}
\sup_{x>0} \Bigl|\P(nM_{1,n}^{(r)}<x)-
\exp\Bigl(-n\Bigl(\sum\nolimits_{i=0}^m\! \theta_i f_{k_i,r}(c_i x)\Bigr)\Bigr)
\Bigr|\to 0
\end{equation}
as $n\to\infty$, where $\theta_i=
c_i(x_{i+1}-x_{i})$, $k_i=[n\theta_i]$, $i=0,\ldots,m$.
Moreover,
$$
\P(nM_{1,n}^{(r)}-c_{\min}^{-1}b_n^{(r)}-c_{\min}^{-1}\log \theta^*<t)\to \exp(-\exp(-t))
$$
as $n\to\infty$, where $c_{\min}=\min(c_1,\ldots,c_m)$ and
$
\theta^*=\sum_{i \mid p_i=p_{\min}} \theta_i.
$
\end{theorem}

The last assertion of theorem \ref{teo:mix3} was actually given in \cite[section 8]{dkr92}.
The next result deals with bounded PDF defined on $[A,B]$. 

\begin{theorem}
\label{teo:mix1}
Let $X_1$,~\ldots, $X_n$ be a sample from an absolutely continuous distribution
with the support $[A,B]$, where $-\infty<A<B<\infty$, and having a PDF $p(x)$; $T=\bigcup_{j=1}^{q} I_j=\{x \mid p(x)>0\}\subseteq [A,B]$ and $I_j$ be the intervals; $p_{\min}=\essinf_{x\in T} p(x)$; $p_+=p_{\min}+\epsilon$ for some $\epsilon>0$; $\widetilde I_{j}=\{x\in I_j \mid p(x)<p_+\}$, $j=1,\ldots,q$. Assume that $p_{\min}>0$ and $p(x)$ satisfies the H\"older's condition locally
\begin{equation}
\label{equ:Holder}
|p(x)-p(y)|\leq C |x-y|^{\alpha}\quad\mbox{for all}\quad x,y\in \widetilde I_{j},
\end{equation}
$j=1,\ldots,q$, with some $C$, $\alpha\!>\!0$ and some $\epsilon>0$.
Then
\begin{equation}
\label{equ:gform2}
\sup_{x>0}\Bigl|\P(nM_{1,n}^{(r)}<x)-\exp\Bigl(-n\int_{A}^{B} f_{n,r}(xp(u))
p(u)du\Bigr)\Bigr|\to 0
\end{equation}
as $n\to\infty$.
\end{theorem}

Good approximation in the particular case $r=1$ due to Barbe \cite{brb92} can 
be obtained in the following corollary.

\begin{corollary}
\label{cor:mix1}
Under the assumptions of Theorem \ref{teo:mix1},
$$
\sup_{x>0} \Bigl|\P(nM_{1,n}<x)-\exp\Bigl(-n\int_{A}^{B} e^{-xp(u)}
p(u)du\Bigr)\Bigr|\to 0\quad\mbox{as}\quad n\to\infty.
$$
\end{corollary}

\section{Extended class of distributions}
\label{sec:wide}

In this section, the notation $M_{k,n}^{(r)}$ for the $k$-th maximal value from 
$S_{1,n}^{(r)}$,~\ldots, $S_{n-r,n}^{(r)}$ is used, as it was introduced in 
Section~\ref{sec:intro}. 
Following Deheuvels~\cite{dvs86}, we investigate  convergence in distribution of the maximal $r$-spacings and $k$-th maximal $r$-spacings for the three
domains of attraction of extreme values. Deheuvels \cite{dvs86} obtained limit theorems for classes of distributions related to the three extremal types of distributions. The same classes
will be considered in this section.

This three classes of distributions are defined as follows. Let $A=\inf\{x \mid F(x)>0\} < B=\sup\{x \mid F(x)<1\}$ and assume that
$p(x)>0$ for
all $x\in (A,B)$, $A>-\infty$ and $p(A_+)=\lim_{x\to A_+}p(x)>0$. Denote by
$G(u)=\inf\{x \mid 1-F(x)\leq u\}$ the inverse of survival function and set
$b_n=G(1/n)$ and $a_n=G(1/(ne))-G(1/n)$. 

{\it Gumbel type}. For all $x\in (-\infty,\infty)$,
\begin{equation}
\label{equ:Gumbel}
\lim_{n\to \infty} \P(a_n^{-1}(X_{n,n}-b_n)\leq x)=\lim_{n\to\infty}
F^{n}(a_n x+b_n)= e^{-e^{-x}}=\Lambda(x).
\end{equation}

{\it Fr\'echet type}. For all $x\in (0,\infty)$ and for some
$a>0$,
\begin{equation}
\label{equ:Frechet}
\lim_{n\to \infty} \P(b_n^{-1} X_{n,n}\leq x)=\lim_{n\to\infty}
F^{n}(b_n x)=e^{-x^{-a}}=\Phi_a(x)
\end{equation}

{\it Weibull type.}
For all $x\in (-\infty,0)$, for some $a>0$ and finite $B$,
\begin{equation}
\label{equ:Weibull}
\lim_{n\to \infty} \P((X_{n,n}-B)/(B-b_{n})\leq x)= e^{-(-x)^{a}}=\Psi_a(x).
\end{equation}

All these limit types can be combined and represented by the so called generalized extreme value distribution with CDF
$$
F(x;\xi)=\exp(-(1-\xi x)^{-1/\xi}).
$$
Cases $\xi=0$, $\xi>0$ and $\xi<0$ correspond to Gumbel's,
Fr\'echet and Weibull's extremal types respectively.

The following theorem describes the asymptotic distributions for the $k$-th maximal $r$-spacings associated to the three extremal types distributions.

\begin{theorem}
\label{teo:rspacing}
Let $X_1,\ldots,X_n$ be a sample from an absolutely continuous distribution
with a PDF $p(x)$. Assume that  $p(x)>0$ for all $x\in
(A,B)$,
$A>-\infty$ and $p(A_+)=\lim_{x\to A_+}p(x)>0$. Let $E_1,E_2,\ldots$ be a
sequence of {\em i.i.d.} random variables having the standard exponential
distribution $E(1)$.\\
(i) Assume that the original distribution satisfies (\ref{equ:Gumbel}) and
$p(x)$ is continuous and ultimately nonincreasing in the upper tail. Then
$$
\lim_{n\to\infty} \P(a_n^{-1} M_{k,n}^{(r)}\leq x)=H_{k,r}^{G}(x),
$$
where $H_{k,r}^{G}(x)$ is the distribution function of $k$-th maximum of
$$
\textstyle
\Bigl\{\sum\nolimits_{l=1}^r (j+l-1)^{-1}E_{j+l-1},\;j\geq 0\Bigr\}.
$$
(ii) Assume that the original distribution satisfies (\ref{equ:Frechet})
for some $a>0$. Then
$$
\lim_{n\to\infty} \P(b_n^{-1} M_{k,n}^{(r)}\leq x)=H_{k,r}^{F}(x),
$$
where $H_{k,r}^{F}(x)$ is the distribution function of $k$-th maximum of
$$
\textstyle
\Bigl\{\sum\nolimits_{l=1}^r \bigl((\sum\nolimits_{s=1}^{j+l-1}
E_{s})^{-1/a}-(\sum\nolimits_{s=1}^{j+l} E_{s})^{-1/a}\bigr),\;j\geq 0\Bigr\}.
$$
(iii) Assume that the  original distribution satisfies (\ref{equ:Weibull})
for some $a>1$. Then
$$
\lim_{n\to\infty} \P(M_{k,n}^{(r)}/(B-b_n)\leq x)=H_{k,r}^{W}(x),
$$
where $H_{k,r}^{W}(x)$ is the distribution function of $k$-th maximum of
$$
\textstyle
\Bigl\{\sum\nolimits_{l=1}^r \bigl((\sum\nolimits_{s=1}^{j+l}
E_{s})^{1/a}-(\sum\nolimits_{s=1}^{j+l-1} E_{s})^{1/a}\bigr),\;j\geq 0\Bigr\}.
$$
\end{theorem}

The condition $p(A_+)>0$ of the Theorem~\ref{teo:rspacing} can be weakened in particular cases. For
example, considering symmetric distribution leads to the following corollary.

\begin{corollary}
\label{cor:symmetry}
Let $X_1,\ldots,X_n$ be a sample from an absolutely continuous distribution
with PDF $p(x)$. Assume that $p(x)>0$ for all $x\in
(A,B)$ and $p(A+x)=p(B-x)$ for almost all $x$. \\
(i) Assume that the  original distribution satisfies (\ref{equ:Gumbel}) and
$p(x)$ is continuous and ultimately nonincreasing in the upper tail. Then
$$
\lim_{n\to\infty} \P(a_{[n/2]}^{-1} M_{k,n}^{(r)}\leq x)=(H_{k,r}^{G}(x))^2.
$$
\\
(ii) Assume that the  original distribution satisfies (\ref{equ:Frechet}) for
some $a>0$. Then
$$
\lim_{n\to\infty} \P(b_{[n/2]}^{-1} M_{k,n}^{(r)}\leq x)=(H_{k,r}^{F}(x))^2.
$$
\\
(iii) Assume that the  original distribution satisfies (\ref{equ:Weibull}) for
some $a>1$. Then
$$
\lim_{n\to\infty} \P(M_{k,n}^{(r)}/(B-b_{[n/2]})\leq x)=(H_{k,r}^{W}(x))^2.
$$
\end{corollary}

For applications in the coverage problems, introduce another sample
$Y_1$,~\ldots, $Y_n$ from a positive distribution with CDF $F_Y$ and independent of
$X_1$,~\ldots, $X_n$. Denote by $N_{n,r}(x)$ the number of
of $r$-spacings greater than $x\,Y_i $, {\em i.e.}  $N_{n,r}(x)=\#\{i\in
\{1,\ldots,n^*\} \mid S_{i,n}^{(r)}\geq x\,Y_i \}$. By applying
Theorem~\ref{teo:rspacing}, we have the following result.

\begin{corollary}
\label{cor:Gumbel}
Under the assumptions of Theorem \ref{teo:rspacing}, we have
$$
\P(N_{n,r}(x)\leq k)\toi{n}\int_{0}^{\infty} H_{k,r}^{X}(ux) dF_Y(u).
$$
with $X=G$ under (\ref{equ:Gumbel}), $X=F$ under (\ref{equ:Frechet}) or $X=W$
under (\ref{equ:Weibull}), respectively.
\end{corollary}

\section{Simulations and numerical results}\label{sec:simul}


\subsection{The quality of approximations}\label{sub:simul}

To assess the approximations of the maximal $r$-spacings distribution,
we conducted a simulation study, using the R packages \cite{r11}.

The estimated CDF of $M_{1,n}^{(r)}$ and the empirical cumulative distribution function (ECDF) based on results of $6000$ replicates for
the uniform model with $n=10^4$, $10^5$ and $r=5$ are shown on the figure \ref{fig:fig01}.

\begin{figure}[ht!]
\centering
\includegraphics[height=140pt]{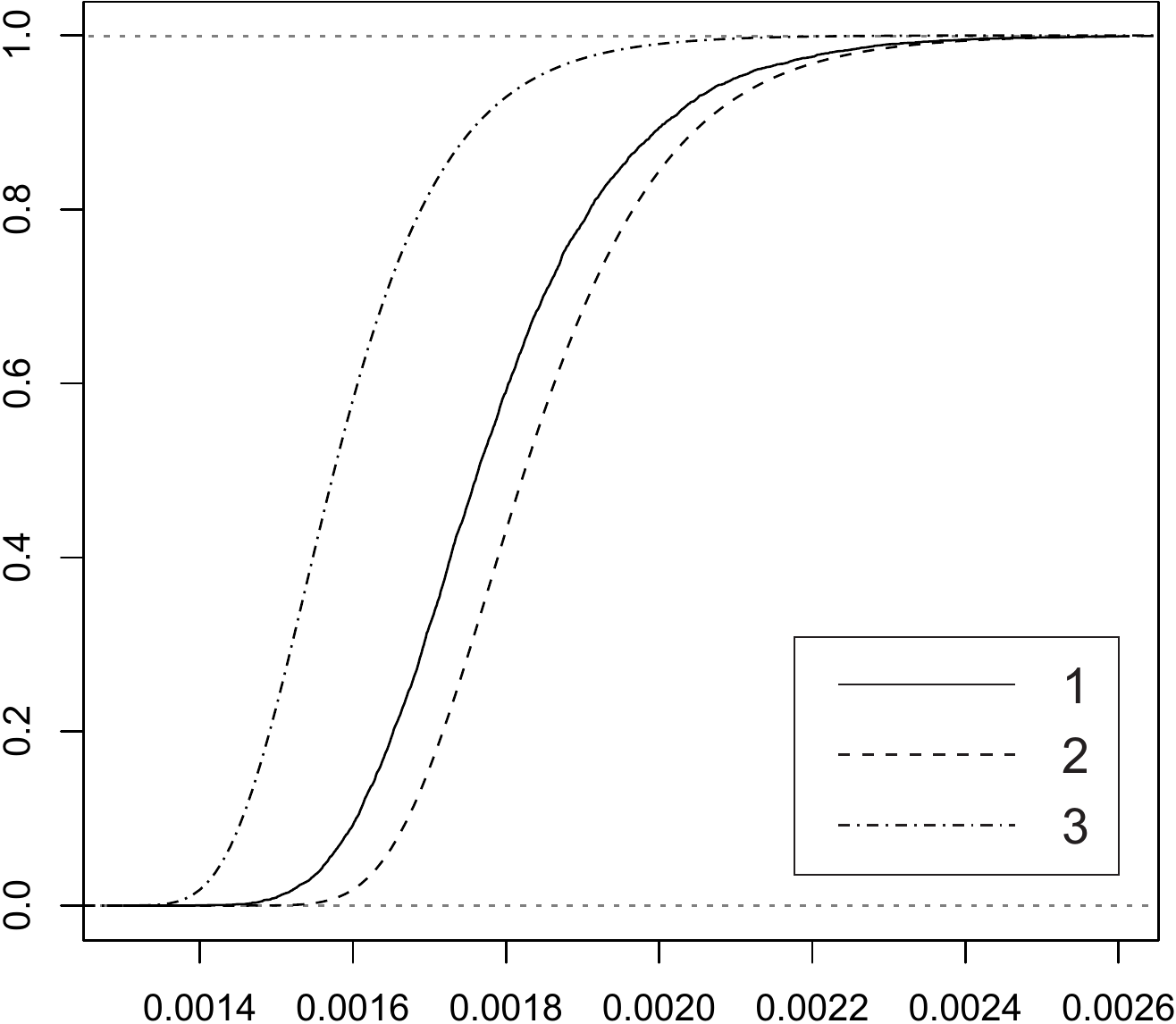}\qquad
\includegraphics[height=140pt]{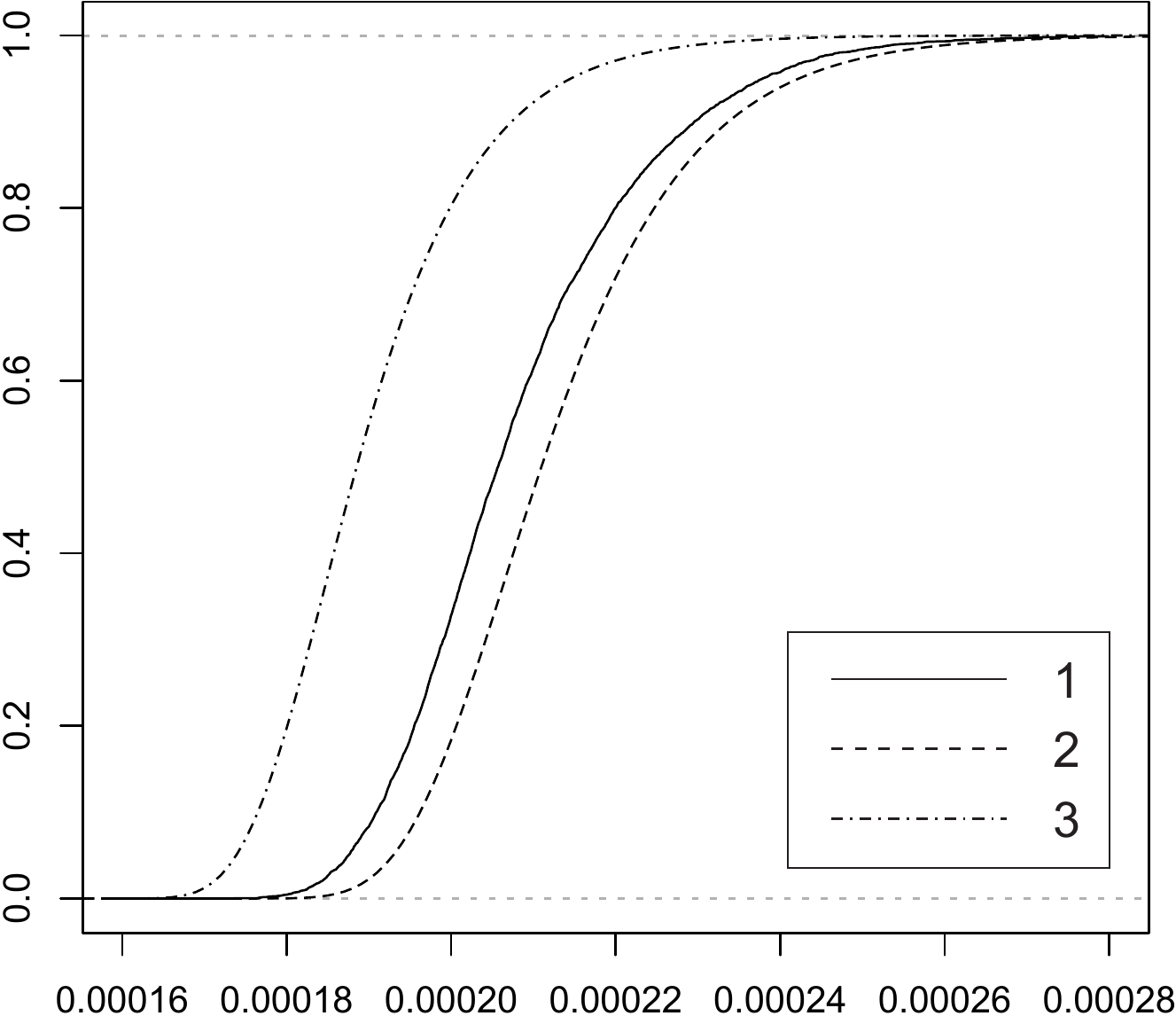}
\caption{\label{fig:fig01} \textit{Uniform spacings model: (1) ECDF; (2) CDF estimator  (\ref{equ:ldappr}); (3) CDF estimator  (\ref{equ:limmaver}) with
$b_n=\log n + (r - 1)\log \log n - \log \Gamma(r)$.}}
\end{figure}

It is important to note that on the figure \ref{fig:fig01}, the curve (2) on both
graphs is close to the empirical CDF of
$\widehat{M}_{1,n}^{(r)}$, so the last is not drawn.

The corresponding estimated quantiles of the distribution
$M_{1,n}^{(r)}$ under $r=5$ as well as the estimated quantiles under $r=1$ are
presented in tables \ref{tabl:tabl01} and \ref{tabl:tabl02}. Rows ``$M_{1,n}^{(r)}$"
and ``$\widehat M_{1,n}^{(r)}$" contain the empirical quantiles of
$M_{1,n}^{(r)}$ and $\widehat M_{1,n}^{(r)}$, $r=1,5$, based on 6000
replicates. The quantiles estimated by (\ref{equ:ldappr}) are contained in row
``$(\ref{equ:ldappr})$''. The estimator based on (\ref{equ:limmaver}) with
$b_n=\log n+(r-1)\log \log n - \log \Gamma(r)$ denoted by
``$\log$''.
\begin{table}[ht]
{\small\tabcolsep=2.5pt
\caption{\label{tabl:tabl01}}
\centerline{\textit{Quantiles for uniform spacings, $n =10^4$. All values must be multiplied by $10^{-3}$.\vspace{+2mm}}}
\centerline{\begin{tabular}{|l|l|l|l|l|l|l|l|l|l|l|}
\hline
& \multicolumn{2}{|c|}{$0.05$} & \multicolumn{2}{|c|}{$0.25$} & \multicolumn{2}{|c|}{$0.5$} & \multicolumn{2}{|c|}{$0.75$} & \multicolumn{2}{|c|}{$0.95$}\\
\hline
$r$&\multicolumn{1}{|c|}{1}&\multicolumn{1}{|c|}{5}&\multicolumn{1}{|c|}{1}&\multicolumn{1}{|c|}{5}& \multicolumn{1}{|c|}{1}&\multicolumn{1}{|c|}{5}&\multicolumn{1}{|c|}{1}&\multicolumn{1}{|c|}{5}&\multicolumn{1}{|c|}{1}&\multicolumn{1}{|c|}{5}\\
\hline
$M_{1,n}^{(r)}$ & 0.8136 & 1.5665 & 0.8912 & 1.669 & 0.9603 & 1.7617 & 1.0497 & 1.8801 & 1.2225 & 2.1021\\
\hline
$\widehat{M}_{1,n}^{(r)}$ & 0.8097 & 1.6376 & 0.8896 & 1.7347 & 0.9582 & 1.8198 & 1.0461 & 1.9322 & 1.2231 & 2.1455\\
\hline
(\ref{equ:ldappr}) & 0.8113 & 1.6377 & 0.8884 & 1.7367 & 0.9577 & 1.8245 & 1.0456 & 1.9345 & 1.2181 & 2.1462\\
\hline
$\log$ & 0.8113 & 1.4293 & 0.8884 & 1.5063 & 0.9577 & 1.5756 & 1.0456 & 1.6636 & 1.2181 & 1.836\\
\hline
\end{tabular}}}
\end{table}

\begin{table}[ht!]
{\small\tabcolsep=2.5pt
\caption{\label{tabl:tabl02}}
\centerline{\textit{Quantiles for uniform spacings, $n =10^5$. All values must be multiplied by $10^{-4}$.\vspace{+2mm}}}
\centerline{\begin{tabular}{|l|l|l|l|l|l|l|l|l|l|l|}
\hline
& \multicolumn{2}{|c|}{$0.05$} & \multicolumn{2}{|c|}{$0.25$} & \multicolumn{2}{|c|}{$0.5$} & \multicolumn{2}{|c|}{$0.75$} & \multicolumn{2}{|c|}{$0.95$}\\
\hline
$r$&\multicolumn{1}{|c|}{1}&\multicolumn{1}{|c|}{5}&\multicolumn{1}{|c|}{1}&\multicolumn{1}{|c|}{5}& \multicolumn{1}{|c|}{1}&\multicolumn{1}{|c|}{5}&\multicolumn{1}{|c|}{1}&\multicolumn{1}{|c|}{5}&\multicolumn{1}{|c|}{1}&\multicolumn{1}{|c|}{5}\\
\hline
$M_{1,n}^{(r)}$ & 1.0413 & 1.8759 & 1.1189 & 1.9708 & 1.1882 & 2.0563 & 1.2744 & 2.1682 & 1.4487 & 2.3803\\
\hline
$\widehat{M}_{1,n}^{(r)}$ & 1.0436 & 1.923 & 1.1181 & 2.0222 & 1.1888 & 2.1059 & 1.2805 & 2.2112 & 1.4508 & 2.4122\\
\hline
(\ref{equ:ldappr}) & 1.0416 & 1.9295 & 1.1186 & 2.0247 & 1.1879 & 2.1096 & 1.2759 & 2.2163 & 1.4483 & 2.4227\\
\hline
$\log$ & 1.0416 & 1.7387 & 1.1186 & 1.8157 & 1.1879 & 1.885 & 1.2759 & 1.973 & 1.4483 & 2.1454\\
\hline
\end{tabular}}}
\end{table}

In the case $r=1$ all of these results work well and all of the above
curves are close (which also can be seen from tables of
quantiles).

Estimators obtained due to the formula (\ref{equ:gform2}) for the truncated normal
distribution with parameters $\mu = 1/2$, $\sigma^2 = 1$, $a = 0$, $b=1$
(lower and upper bounds) are in agreement with the empirical data, which
can be seen on the table \ref{tabl:tabl03}. 

\begin{figure}[ht!]
\begin{center}\includegraphics[height=140pt]{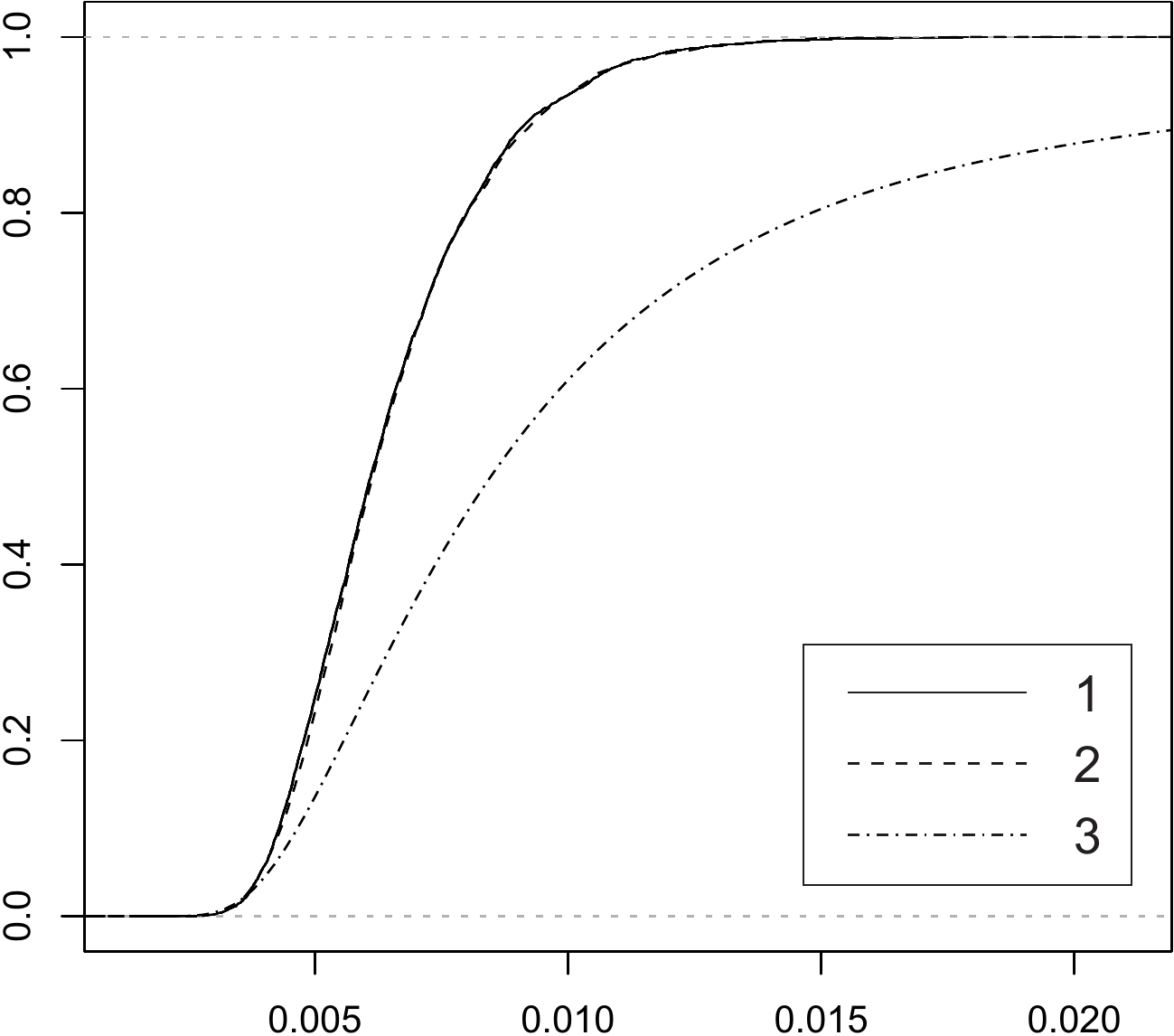}\end{center}
\caption{\label{fig:fig02}\textit{Spacings for the triangle distribution: (1) ECDF; (2) CDF estimator via Corollary \ref{cor:symmetry};(3) CDF estimator (\ref{equ:gform2}).}}
\end{figure}

\begin{table}[ht]
{\small
\caption{\label{tabl:tabl03}}
\centerline{\textit{Quantiles for spacings of the truncated normal distribution, $n =10^4$.}}
\centerline{\textit{All values must be multiplied by $10^{-3}$.\vspace{+2mm}}}
\tabcolsep=2.5pt
\centerline{\begin{tabular}{|l|l|l|l|l|l|l|l|l|l|l|}
\hline
& \multicolumn{2}{|c|}{$0.05$} & \multicolumn{2}{|c|}{$0.25$} & \multicolumn{2}{|c|}{$0.5$} & \multicolumn{2}{|c|}{$0.75$} & \multicolumn{2}{|c|}{$0.95$}\\
\hline
$r$&\multicolumn{1}{|c|}{1}&\multicolumn{1}{|c|}{5}&\multicolumn{1}{|c|}{1}&\multicolumn{1}{|c|}{5}& \multicolumn{1}{|c|}{1}&\multicolumn{1}{|c|}{5}&\multicolumn{1}{|c|}{1}&\multicolumn{1}{|c|}{5}&\multicolumn{1}{|c|}{1}&\multicolumn{1}{|c|}{5}\\
\hline
$M_{1,n}^{(r)}$ & 0.8197 & 1.5782 & 0.8932 & 1.6809 & 0.9646 & 1.7729 & 1.0517 & 1.8876 & 1.2281 & 2.1117\\
\hline
(\ref{equ:gform2}) & 0.8154 & 1.5769 & 0.8931 & 1.682 & 0.9624 & 1.7757 & 1.0515 & 1.8918 & 1.2286 & 2.1145\\
\hline
(\ref{equ:ldappr}) & 0.8149 & 1.6501 & 0.8928 & 1.7512 & 0.9629 & 1.841 & 1.052 & 1.9537 & 1.2271 & 2.1711\\
\hline
\end{tabular}}}
\end{table}

The first row ``$M_{1,n}^{(r)}$" in table \ref{tabl:tabl03} displays
the empirical quantiles of $M_{1,n}^{(r)}$, $r=1,5$, based on $6000$
replicates. Row ``(\ref{equ:gform2})'' was obtained by using formula
(\ref{equ:gform2}) via empirical distribution function of uniform spacings 
based on $10^4$ replicates and row ``(\ref{equ:ldappr})'' was obtained via
the same formula (\ref{equ:gform2}) but utilizing the asymptotic
(\ref{equ:ldappr}).

The estimator given by formula (\ref{equ:gform2}) for the
triangle distribution having the PDF (\ref{equ:trzd}) with $\kappa=1/2$ is
shown on the figure \ref{fig:fig02} (number of replicates is $6000$, $n=10^4$,
$r=1$) and the corresponding quantiles are given in table \ref{tabl:tabl04}. It
is clear that (\ref{equ:gform2}) is not applicable for this distribution.
Much satisfactory results can be obtained by using Corollary \ref{cor:symmetry} (see
(\ref{equ:trzs}) for the explicit formula). The corresponding curve on the
figure \ref{fig:fig02} is close to the ECDF curve and the
quantiles in table \ref{tabl:tabl04} display good approximation under $r=1$ and
$r=5$ with $n=10^4$.
\vspace{-5mm}

\begin{table}[ht!]
{\small
\caption{\label{tabl:tabl04}}
\centerline{\textit{Quantiles for spacings of the triangle distribution,
$n =10^4$.}}
\centerline{\textit{All values must be multiplied by $10^{-2}$.\vspace{2mm}}}
\tabcolsep=2.5pt
\centerline{
\begin{tabular}{|l|c|c|c|c|c|c|c|c|c|c|}
\hline
& \multicolumn{2}{|c|}{$0.05$} & \multicolumn{2}{|c|}{$0.25$} & \multicolumn{2}{|c|}{$0.5$} & \multicolumn{2}{|c|}{$0.75$} & \multicolumn{2}{|c|}{$0.95$}\\
\hline
$r$&\multicolumn{1}{|c|}{1}&\multicolumn{1}{|c|}{5}&\multicolumn{1}{|c|}{1}&
\multicolumn{1}{|c|}{5}&
\multicolumn{1}{|c|}{1}&\multicolumn{1}{|c|}{5}&\multicolumn{1}{|c|}{1}
&\multicolumn{1}{|c|}{5}&\multicolumn{1}{|c|}{1}&\multicolumn{1}{|c|}{5}\\
\hline
$M_{1,n}^{(r)}$ & 0.3918 & 0.9481 & 0.5027 & 1.1506 & 0.6125 & 1.331 & 0.753 &
1.5406 & 1.0466 & 1.8828 \\
\hline
C.\ref{cor:symmetry} & 0.3942 & 0.9467 & 0.5085 & 1.1508 & 0.6145 & 1.3406 & 0.7538 & 1.5496 &
1.0405
& 1.8902\\
\hline
(\ref{equ:gform2}) & 0.4085 & $-$ & 0.6005 & $-$ & 0.8474 & $-$ & 1.3039 & $-$
&
3.6956 & $-$\\
\hline
\end{tabular}}
}
\end{table}

\subsection{Coverage problems}
\label{sub:covapp}

Let $L$ be a length of single read in bp, $N$ be a length of genome, $I$ be a
minimal overlap of two random reads to be assembled in one sequence. As we see above, the
full genome $r$-times coverage problem can be transformed to the problem of
$r$-times coverage of the interval $[0,1]$ by random intervals of length
$l=(L-I)/N$ and
$$
\P(Q_r)=\P(M_{1,n}^{(r)}<l),
$$
where $Q_r$ denotes the following event: all existing reads contains at least $r$ bases from all
positions, taking into account only reads with minimal intersection $I$ bp with
at least $r$ neighbors.

The simplest case for applications in coverage problem is the uniformly distributed
random variables determining locations of the intervals considered in Section~\ref{sec:unif}.
All necessary results for uniform case were posed in Section~\ref{sec:unif}.
For example, to cover the whole human genome of length $3.2\times 10^9$ under
uniform distribution with reads of length $200$ and overlap $I=50$ with the 
probability not less then $95\%$ one needs to acquire at least $5\times 10^8$ random reads.

Coverage in practice is typically not uniform. Note that all distributions with
bounded support densities $0<\delta<p(x)$ for all $x\in [A,B]$ belong to
Weibull's extremal type with $a=1$. The asymptotic results for this case are
given in Section~\ref{sec:mixed}.

Results of Theorem \ref{teo:rspacing} with Corollary \ref{cor:symmetry} are
applicable for some symmetric distributions with $p(1_-)=p(0_+)=0$. For example,
the trapezoidal distribution having PDF
\begin{equation}
\label{equ:trzd}
p(x)=
\begin{cases}
(\kappa(1-\kappa))^{-1} x,\; x\in [0,\kappa);\cr
(1-\kappa)^{-1}, \; x\in [\kappa,(1-\kappa));\cr
(\kappa(1-\kappa))^{-1} (1-x),\; x\in [1-\kappa,1];\cr
0, \; x\notin [0,1], \cr
\end{cases}
\end{equation}
with some $\kappa\in [0,1/2]$ belongs to Weibull's extremal type
with $a=2$ (see \cite{llr83}, Theorem 1.6.1).
Then, by Corollary \ref{cor:symmetry} (iii),
\begin{equation}
\label{equ:trzs}
\P(M_{k,n}^{(r)}\leq l)\approx (H_{1,r}^{W}(l\sqrt{n}/\sqrt{2\kappa(1-\kappa)}\,))^2,
\end{equation}
where $H_{1,r}^{W}$ is the distribution function of 
$$\max_{j\geq 0}\biggl(\sum_{l=1}^r
\biggl(\biggl(\sum_{s=1}^{j+l} E_{s}\biggr)^{\!1/2}-\biggl(\sum_{s=1}^{j+l-1} E_{s}\biggr)^{\!1/2}\biggr)\biggr),$$
 where
$E_1,E_2,\ldots$ be a sequence of {\em i.i.d.} random variables having the standard
exponential $E(1)$ distribution. By Lemma \ref{lem:loc5} the same approximation
is valid under any PDF having the same behavior near bounds $0$ and $1$ and
separated from zero in other points of the interval $[0,1]$.
Corollary \ref{cor:symmetry} (i) will be used under
$
p(x)=e^{-1/x}
$
in some neighborhood of $0$ and
$
p(x)=e^{-1/(1-x)}
$
in some neighborhood of $1$.

The important extension of the coverage problem arises if we allow the
subsegments covering the initial interval to have a random lengths. Under the independence assumption of the
original sample of locations $X_1$, \ldots, $X_n$ and the corresponding sample of
covering small segments length $Y_1$, \ldots, $Y_n$
the $r$-times coverage probability can be obtained by $\P(M_{1,n}^{(r)}<Y)$,
where $Y$ has the distribution of random length of the small segments concentrated on
$[0,1]$ and independent of $M_{1,n}^{(r)}$.  Then the probability of $r$-times
coverage of whole interval $[0,1]$ is
\begin{equation}
\label{equ:randcor}
\P(M_{1,n}^{(r)}<Y)=\int_{0}^{1} \P(M_{1,n}^{(r)}<s) dF_{Y}(s),
\end{equation}
where $F_{Y}$ is the distribution function of $Y$. Applying the approximate
formulas for the distribution function of $M_{1,n}^{(r)}$ we obtain the required
probability.  Analogously, the probability to have less then $k$ regions without
$r$-times coverage is given by
$$
\P(N_{n,r}<k)=\P(M_{k,n}^{(r)}<Y)=\int_{0}^{1} \P(M_{k,n}^{(r)}<s) dF_{Y}(s)
$$
and asymptotic formulas for $\P(M_{k,n}^{(r)}<s)$ are obtained in previous sections.

Based on the approximation formulas for uniform and nonuniform spacings we fulfill table
\ref{tabl:tabl05} containing total (or expected under random length of reads)
length of reads necessary to $r$-times full coverage of Human's genome  with
probability at least $0.95$ and under $I=50$.
In columns ``Random'' we list numerical results for the normal case $N(300,50^2)$.
Remark that the results presented in table \ref{tabl:tabl05} can't be obtained
and even verified by simulations in closed time because of too much length of
Human's genome.

\begin{table}[ht!]
{\small
\caption{\label{tabl:tabl05}}
\centerline{\textit{Total (expected) length of reads / whole genome length required for full }}
\centerline{\textit{$r$-coverage of Human's genome with 95\% probability under $I=50$.\vspace{+2mm}}}
\centerline{\tabcolsep=2pt
\begin{tabular}{|l|c|c|c|c|c|c|c|c|c|c|c|c|}
\hline
Distribution &\multicolumn{4}{c|}{Uniform} & \multicolumn{4}{c|}{Truncated $N(1/2,1)$} & \multicolumn{4}{c|}{Truncated $N(1/2,1/4)$}  \\
\hline
$L$ & $100$ & $200$ & $300$ & Random & $100$ & $200$ & $300$ & Random & $100$ & $200$ & $300$ & Random \\
\hline
$r=1$ & 48 & 31 & 27 & 27 & 49 & 31 & 27 & 27 & 173 & 109 &  95 &  95 \\
\hline
$r=2$ & 55 & 35 & 31 & 31 & 56 & 36 & 32 & 32 & 201 & 127 & 112 & 112 \\
\hline
$r=5$ & 72 & 46 & 41 & 41 & 73 & 47 & 42 & 42 & 268 & 171 & 151 & 151 \\
\hline
$r=10$ & 94 & 61 & 54 & 54 &  96 &  62 &  55 &  55 & 359 & 231 & 204 & 204 \\
\hline
$r=50$ & 227 & 149 & 133 & 133 & 237 & 155 & 138 & 138 & 916 & 599 & 534 & 534 \\
\hline
\end{tabular}}
}
\end{table}

The estimator (\ref{equ:ldappr}) is used to fulfill the table for the uniform
case. The estimators (\ref{equ:ldappr}) and (\ref{equ:gform2}) are used to
fulfill the table for the truncated normal cases. The integration
(\ref{equ:randcor}) is applied in the case of random length of reads with
$Y\sim N(300,50^2)$ and the same formulas (\ref{equ:ldappr}) and 
(\ref{equ:gform2}) are used to evaluate $\P(M_{1,n}^{(r)}<s)$ under the integral. Remark the total correspondence of
results for random and non-random length of reads with the same expected
values.


\section{Proofs}
\label{sec:proofs}

The following proposition were addressed in Section~\ref{sec:unif}.

\begin{proposition}
\label{pro:inda}
Let $E_1$, $E_2$, \ldots, be a sequence of {\em i.i.d.} random variables having the standard exponential distribution $E(1)$, $S_j^{(r)} = \sum\limits_{i=j}^{j+r-1} E_i$, $j=1, 2, \ldots$ having 
$\Gamma(r, 1)$ distribution. Set $Q_j = \{S_j^{(r)} < x\}$. Then,
\begin{equation}
\label{equ:p1}
\P(Q_1,\ldots,Q_n)\geq \P(Q_1)\cdots \P(Q_n)
\end{equation}
for all $n\in\mathbb{N}$.
\end{proposition}

\begin{proof} To prove this inequality it is sufficient to prove that
$$
\P(Q_1,\ldots,Q_{k-1} \mid S_k^{(r)}<y)\geq \P(Q_1,\ldots,Q_{k-1})
$$
for all $k$ and $y$. Remark that
$$
\P(Q_1,\ldots,Q_{k-1}\mid S_k^{(r)}<\infty)=\P(Q_1,\ldots,Q_{k-1})
$$ so if we will be able to prove that
$$f(y)=\P(Q_1,\ldots,Q_{k-1}\mid S_k^{(r)}=y)$$
is a monotone decreasing function of $y$ we will obtain (\ref{equ:p1}).

Set $D_r(y)=\{(z_1,\ldots,z_{r-1})\in\Rb^{n-1} \mid 0<z_1<\ldots<z_{r-1}<y\}$. Then,
$$
f(y) = \int_{D_r(y)}\P(Q_1,\ldots,Q_{k-1} \mid S_k^{(1)} = z_1, \ldots, S_k^{(r-1)} = z_{r-1}, S_k^{(r)}=y)\times {}
$$
$$
{}\times p_{S_k^{(1)}, \ldots, S_k^{(r-1)} \mid S_k^{(r)} = y}(z_1, \ldots,z_{r-1})\,dz_1\,\ldots\,dz_{r-1},
$$
where the conditional PDF $p_{S_k^{(1)}, \ldots, S_k^{(r-1)} \mid S_k^{(r)} = y}$ is given by
$$
p_{S_k^{(1)}\!\!\!, \ldots, S_k^{(r-1)} \mid S_k^{(r)} = y}(z_1, \ldots,z_{r-1}) = \frac{p_{S_k^{(1)}\!\!\!, \ldots, S_k^{(r-1)}\!\!\!, S_k^{(r)}}(z_1, \ldots, z_{r-1}, y)}{p_{S_k^{(r)}}(y)}
$$
$$
\!\!\!\!\!\!\!\!=
\frac{p_{E_{k}, \ldots, E_{k+r-1}}(z_1, z_2 - z_1, \ldots, z_{r-1} -z_{r-2}, y - z_{r-1})}{ p_{S_k^{(r)}}(y)}
$$
$$
= \frac{e^{-y}}{p_{S_k^{(r)}}(y)}\,\indi_{\{0<z_1<\ldots<z_{r-1}<y\}} =  \frac{\Gamma(r)}{y^{r-1}}\,\indi_{\{0<z_1<\ldots<z_{r-1}<y\}}.
$$
We see that
$$
\P(Q_1,\ldots,Q_{k-1}\mid S_k^{(1)} = z_1, \ldots, S_k^{(r-1)} = z_{r-1}, S_k^{(r)}=y)
$$
$$
= \P(S_1^{(r)} < x, \ldots, S_{k-r}^{(r)} < x, S_{k - r + 1}^{(r - 1)} < x - z_1, \ldots, S_{k-1}^{(1)} < x - z_{r-1})
$$
$$
= g(z_1, \ldots, z_{r-1})
$$
is monotone decreasing function of every $z_i$, $i = 1, \ldots, r-1$, therefore
\begin{eqnarray*}
f(y)&=& \frac{\Gamma(r)}{y^{r-1}}\; \int_{D_r(y)} g(z_1, \ldots, z_{r-1})\,dz_1\,\ldots\,dz_{r-1} \\
&=& \int_{D_r(y)} g(z_1, \ldots, z_{r-1})\,dz_1\,\ldots\,dz_{r-1}\Bigm/
\int_{D_r(y)}dz_1\,\ldots\,dz_{r-1}
\end{eqnarray*}
is monotone decreasing function of $y$. The proposition is proved.
\end{proof}

To prove the Theorem \ref{teo:mix3} we need several ancillary results.

\begin{lemma}
\label{lem:loc2}
Let $U_1,\ldots,U_n$ the a sample from the standard uniform distribution $U(0,1)$;
$S_{i,n}^{(r)}$ be the corresponding uniform $r$-spacings and $M_{k,n}^{(r)}$
be the $k$-th maximal $r$-spacings; $\sigma_{k,n}^{(r)}$ is such that
$M_{k,n}^{(r)}=U_{\sigma_{k,n}^{(r)}+r,n}-U_{\sigma_{k,n}^{(r)},n}$. Then,
$$
\P\Bigl(a\in [U_{\sigma_{k,n}^{(r)},n},U_{\sigma_{k,n}^{(r)}+r,n}]\Bigr)\to 0
\quad\mbox{as}\quad n\to\infty.
$$
\end{lemma}

\begin{proof}
The exchangeable property of the uniform spacings (see e.g. \cite{pyk65}) implies the exchangeable property of the uniform $r$-spacings and, therefore,
$\P(\sigma_{k,n}^{(r)}=s)=1/(n-r+2)$, $s=1,\ldots,n-r+2$.
Let $\delta>0$, $a_{\delta_-}=a-\delta$, $a_{\delta_+}=a+\delta$ and
$A_{\delta}=[0,a_{\delta-}]\cup [a_{\delta+},1]$. By convergence property (in
probability and almost sure) of sample quantiles $U_{[np]+1,n}\to p$ as $n\to\infty$ for all $p\in
(0,1)$,
$$
\P(U_{[n(a_{\delta-}-\delta)]-r,n}\leq
a_{\delta-},U_{[n(a_{\delta+}+\delta)]+1,n}
\geq a_{\delta+} )\to 1 \quad\mbox{as}\quad n\to\infty.
$$
On the other hand,
\begin{eqnarray*}
\P\Bigl(\sigma_{k,n}^{(r)}\in \Bigl[[n(a_{\delta-}&-&\delta)]-r+1,[n(a_{\delta+}+
\delta)]\Bigr]\Bigr)\\
&=&
\frac{[n(a_{\delta+}+\delta)]-[n(a_{\delta-}-\delta)]+r}{n-r+2}\to 4\delta
\quad\mbox{as}\quad n\to\infty.
\end{eqnarray*}
Then,\vspace{-3mm}
$$
\P\Bigl(a\in [U_{\sigma_{k,n}^{(r)}},U_{\sigma_{k,n}^{(r)}+r}]\Bigr)<5\delta
$$
for any fixed $\delta>0$ and sufficienlty large $n$.
\end{proof}

Let $A=x_0<\ldots<x_{m+1}=B$, $J_i=[x_{i},x_{i+1}[$, $i=1,\ldots,n-1$, $J_m=[x_{m},x_{m+1}]$, and $c_i$ are positive constants, $i=0,\ldots,m$.
The following lemma is an extension of Lemma \ref{lem:loc2}.

\begin{lemma}
\label{lem:loc3}
Let $X_1,\ldots,X_n$ be a sample from an absolutely continuous distribution
having PDF $p$:
\begin{equation}
\label{equ:stden}
p(x)=
\begin{cases}
c_i, \quad\mbox{for}\quad x\in J_i, \; i=0,\ldots,m;\cr
0,\quad\mbox{otherwise};\cr
\end{cases}
\end{equation}
$M_{k,n}^{(r)}$ be the $k$-th maximal $r$-spacing and $\sigma_{k,n}^{(r)}$ is
such that $M_{k,n}^{(r)}=X_{\sigma_{k,n}^{(r)}+r,n}-X_{\sigma_{k,n}^{(r)},n}$.
Then,\vspace{-3mm}
$$
\P\Bigl(x\in [X_{\sigma_{k,n}^{(r)},n},X_{\sigma_{k,n}^{(r)}+r,n}]\Bigr)\to 0
\quad\mbox{as}\quad n\to\infty
$$
for any fixed $x\in [A,B]$.
\end{lemma}

\begin{proof}
Let $c_{\min}=\min(c_0,\ldots,c_m)$, and $\tilde i$ is the corresponding index, $d=$ $\max(c_0,\ldots,c_m)/c_{\min}$ and $\delta>0$ is a small positive number.
Remark that $c_{\min}(x_2-x_1)=F(x_2)-F(x_1)\leq d (F(x_4)-F(x_3))$ for any $x_1,x_2\in J_{\tilde i}$ and $x_3,x_4\in [A,B]$ such that $x_2-x_1=x_4-x_3>0$. Therefore,
\begin{eqnarray}
\label{equ:symmetry}
\P\Bigl([X_{\sigma_{k,n}^{(r)},n},X_{\sigma_{k,n}^{(r)}+r,n}]&\subseteq & V_x(\delta)\Bigr)=
\P\Bigl([X_{\sigma_{k,n}^{(r)},n},X_{\sigma_{k,n}^{(r)}+r,n}]\subseteq V_y(\delta)\Bigr)
\\
&\geq &
\P\Bigl([X_{\sigma_{k,n}^{(r)},n},X_{\sigma_{k,n}^{(r)}+r,n}]\subseteq V_z(\delta/d)\Bigr),
\nonumber
\end{eqnarray}
for all $x,y$ such that $V_x(\delta)\cup V_y(\delta)\subseteq J_{\tilde i}$ and $z$ so that $V_z(\delta/d)\subseteq [A,B]$, where $V_u(\delta)$ is the $\delta$-neighbourhood of $u\in [A,B]$. Let $\delta>0$ and
$x_{\tilde i}+\delta\leq y_1<\ldots<y_s\leq x_{\tilde i+1}-\delta$ be such that $y_{i+1}-y_{i}>2\delta$. Then, $V_{y_i}(\delta)\cap V_{y_j}(\delta)=\emptyset$ for all $i\not= j$ and by (\ref{equ:symmetry}),
$$
\sum_{i=1}^k \P\Bigl([X_{\sigma_{k,n}^{(r)},n},X_{\sigma_{k,n}^{(r)}+r,n}]\!\subseteq\! V_{y_i}(\delta)\Bigr)\!=\!s\, \P\Bigl([X_{\sigma_{k,n}^{(r)},n},X_{\sigma_{k,n}^{(r)}+r,n}]\!\subseteq\! V_{y_1}(\delta)\Bigr) \!\leq\! 1.
$$
Choosing $s\to\infty$ as $\delta\to 0$ and using (\ref{equ:symmetry}), we obtain that
\begin{equation}
\label{equ:c1}
\P\Bigl([X_{\sigma_{k,n}^{(r)},n},X_{\sigma_{k,n}^{(r)}+r,n}]\subseteq V_y(\delta)\Bigr)\to 0
\end{equation}
as $\delta\to 0$ for all $y\in (x_{\tilde i},x_{\tilde i+1})$. The result follows immediately.
\end{proof}

Now we are proceeding to prove Theorem \ref{teo:mix3}.

\begin{proof}[Proof of Theorem \ref{teo:mix3}]
Represent the original distribution as
a mixture of uniforms $\sum_{i=0}^m \theta_i U(x_i,x_{i+1})$ singular with each other, where
$\theta_i=c_i\alpha_i$, $\alpha_i=x_{i+1}-x_i$, $i=0,\ldots,m$. By
(\ref{equ:apprf}),
$$
\P(k M_{i,1,k}^{(r)}<x)=\exp(-k f_{n,r}(x/\alpha_i))+o(1),
$$
with $o(1)\to 0$ as $k\to\infty$ uniformly in $x\geq 0$, where $M_{i,1,k}$
are the maximal $r$-spacings based on the sample form $U(x_i,x_{i+1})$,
$i=0,\ldots,m$.

Let $D_i=\{j \mid x_i<X_{j,n}<X_{j+r,n}<x_{j+1}\} $ and $D=\bigcup\nolimits_{i=0}^m D_i$.
Introduce restricted maximal spacings $\breve M_{i,1,n}^{(r)}=\max\{S_{j,n},
j\in D_i; X_{l_i,n}-x_i; x_{i+1}-X_{u_i,n}\}$, where $l_i$ and $u_i$ are the
minimal and the maximal elements of $D_i$ respectively, and $\breve
M_{1,n}^{(r)}=\max(\breve M_{0,1,n}^{(r)},\ldots,\breve M_{m,1,n}^{(r)})$.
Remark that the distribution of the restricted maximal $r$-spacing conditionally
on $Q_{D}=\{\# D_i=n_i,\:i=0,\ldots,m\}$ is given by
$$
\P(\breve M_{1,n}^{(r)}<t|Q_D)=\prod_{i=0}^m \P(\breve M_{i,1,n}^{(r)}<t|Q_D)=
\prod_{i=0}^m \P(M_{i,1,n_i}^{(r)}<t).
$$
Then,
\begin{eqnarray}
\label{equ:mix5}
\P(n \breve M_{1,n}^{(r)}<x|Q_D)&=&\exp\Bigl(-\sum\nolimits_{i=0}^m n_i f_{n_i,r}
\Bigl(\frac{n_i x}{n\alpha_i}\Bigr) \Bigr)+o(1)
\\
&=&\exp\Bigl(-n\sum\nolimits_{i=0}^m \frac{n_i}{n} f_{n_i,r}\Bigl(\frac{n_i
x}{n\alpha_i}\Bigr)
\Bigr)+o(1),\nonumber
\end{eqnarray}
where $o(1)\to 0$ as $n_i\to\infty$ for all $i=0,\ldots,m$, uniformly in $x\geq 0$.
Taking into account (\ref{equ:mix2}), $f_{n_i,r}$ in the last equation can be
changed
by $f_{k_i,r}$, $i=0,\ldots,m$. Then, by $\# D_i/n\to_P\theta_i=c_i\alpha_i$
 as $O_P(1/\sqrt{n})$,
\begin{equation}
\label{equ:mix6}
\P(n \breve M_{1,n}^{(r)}\!<\!x)=\exp\Bigl(-n\sum\nolimits_{i=0}^m \theta_i
f_{k_i,r}(c_i x)
\Bigr)+o(1).
\end{equation}
By Lemma \ref{lem:loc3}, $\P(M_{1,n}=\breve M_{1,n})\to 1$ as $n\to\infty$ and
(\ref{equ:mix3}) is proved.

Under $f_n^{(r)}(x)=1-G_r(x)$,
$$
\P(n M_{1,n}^{(r)}<x)=\exp\Bigl(-n\sum\nolimits_{i=0}^m \theta_i (1-G_r(c_i x))\Bigr)+o(1).
$$
After substitution, $t=p_{\min}x-b_n^{(r)}-\log\theta^*$,
$$
\P(np_{\min} M_{1,n}^{(r)}-b_n^{(r)}-\log\theta^*<t)
$$
$$
=\exp\Bigl(-\sum\nolimits_{i=1}^{n}\exp\bigl(-c_i c_{\min}^{-1}(t+\log n+(r-1)\log\log n-\log\Gamma(r)+\log\theta^*)+{}
$$
$$
+\log n+\log\theta_i\bigr)\times
\sum_{j=1}^{r}\bigl(c_ic_{\min}^{-1}(t+b_n^{(r)}+\log\theta^*)\bigr)^j/\Gamma(j)
\Bigr)+o(1)
$$
$$
\to
\exp(-\exp(-t))\qquad\mbox{as}\qquad n\to\infty.
$$
The Theorem \ref{teo:mix3} is proved.
\end{proof}

\begin{remark}\rm\mbox{\!\!\!}
\label{rem:1}
(i). By (\ref{equ:mix2}) the functions $f_{n_i,r}$ in (\ref{equ:mix3}) can be
changed by~$f_{r}$. 

(ii). It is easy to see that limit behavior of $M_{1,n}^{(r)}$ will be
different in general even if the original distributions are close. For example,
$$
\lim_{n\to\infty} \sup_{x} |\P(nM_{U,1,n}^{(r)}<x)-\P(n c M_{U,1,n}^{(r)}<x)|=1,
$$
for any $c\not=1$, even if $c$ is close to 1.
To prove Theorem \ref{teo:mix1} it will be important to extend the result of Theorem \ref{teo:mix3} to a series of experiments with increasing $m$ as $n\to\infty$.
\end{remark}

The essential part of Theorem \ref{teo:mix1} is the following lemma.

\begin{lemma}
\label{lem:mix1}
Let under conditions of Theorem \ref{teo:mix1} \ $p(x)$ is bounded by a positive constant $M$ and satisfies (\ref{equ:Holder}) for all $x,y\in [A,B]$. Then (\ref{equ:gform2}) holds.
\end{lemma}

\begin{proof}
Without loss of generality assume $[A,B]=[0,1]$. Let $I_1,\ldots,I_m$ be a
partition of the interval $[0,1]$ by intervals $I_i=(x_{i},x_{i+1}]$,
$i=1,\ldots,n$; $0=x_0<x_1<\ldots<x_{m+1}=1$;  $F(x)=\int_{0}^x p(u)du$ be the
CDF corresponding to $p$. Introduce
$c_i=\inf_{x\in I_{i}} p(x)$, $y_0=0$,
\begin{equation}
\label{equ:incr}
y_{i+1}=y_i+c_i^{-1}\int_{x_i}^{x_{i+1}} p(t)dt,
\end{equation}
$J_i=(y_i,y_{i+1}]$ for $i=0,1,\ldots,m$,  and the step-wise PDF
$$
p_{m}(x)=
\begin{cases}
c_i, \quad\mbox{for}\quad x\in J_i, \; i=0,\ldots,m;\cr
0,\quad\mbox{otherwise}.\cr
\end{cases}
$$
Set $F_m(y)=\int_0^y p_m(t)dt$ for all $y>0$ is the corresponding CDF
Then, $F_m(y_i)=F(x_i)$ for all $i=0,\ldots,m+1$. Denote, $G_{m}(x)=F_m^{-1}(F(x))$ for all $x\in [0,1]$. It is easy to verify that $p(x)\geq p_n(G_m(x))$ and, therefore, $G_m(x_2)-G_m(x_1)\geq x_2-x_1$ for any $0\leq x_1<x_2\leq 1$. It means that
$$
\P(M_{1,n}^{(r)}<x)\geq \P(M_{1,n}^{+(r)}<x),
$$
where $M_{1,n}^{+(r)}$ is the maximal $r$-spacing based on a sample from the absolutely continuous distribution having PDF $p_m$.  By (\ref{equ:incr})
$$
y_{i}=\sum\nolimits_{j=0}^{i-1} c_j^{-1}\int_{x_j}^{x_{j+1}} p(t)dt=\sum\nolimits_{j=0}^i c_j^{-1} p(\tilde x_j)(x_{j+1}-x_j)
$$
for some $\tilde x_j\in I_j$, $j=0,\ldots,m$, and
$$
y_{i}-x_{i}=\sum\nolimits_{j=0}^{i-1} (c_j^{-1} p(\tilde x_j)-1)(x_{j+1}-x_j)=
\sum\nolimits_{j=0}^{i-1} c_j^{-1} (p(\tilde x_j)-c_j)(x_{j+1}-x_j).
$$
Then, by (\ref{equ:Holder})
\begin{equation}
\label{equ:pcv}
\max_{i=0,\ldots,m+1} |y_i-x_i|=y_{m+1}-1=O(m^{-\alpha})
\end{equation}
as $\max_{i=0,\ldots,m} (x_{i+1}-x_i)=O(1/m)$ as $m\to\infty$. Consider,
$$
\Delta_{n,m}^{(r)}(x)=n\Bigl(\int_{I_n} f_r(xp_m(t))p_m(t)dt - \int_{I} f_r(xp(t))p(t)dt\Bigr)
$$
$$
=n\int_{I_m} \!\!(f_r(xp_m(t))-f_r(xp(t)))p_m(t)dt+
n\int_{I_m}\!\! f_r(xp(t))(p_m(t)-p(t))dt
$$
$$
=:\Delta_{1,n,m}^{(r)}(x)+\Delta_{2,n,m}^{(r)}(x).
$$
For any fixed $\delta, M$ such that $0<\delta<M<\infty$ introduce the class ${\cal D}(\delta,M)$ of sequences $(x_n)_{n\in\Nb}$ such that $n\int_{I_m}\! f_r(x_n p(t))p(t) dt\in [\delta,M]$.
Remark that
$$
\Delta_{2,n,m}^{(r)}(x_n)=
n\int_{I_m}\! f_r(x_n p(t))p(t) (p_m(t)/p(t)-1)dt \to 0
$$
for any sequence $(x_n)\in {\cal D}(\delta,M)$, $m=m(n)\to\infty$ as $n\to\infty$,  and taking into account (\ref{equ:pcv}) and $x_n=O(\log n)$, we conclude that
$$
\Delta_{1,n,m}^{(r)}(x_n)= n \int_{I} f_r(x_n p_m(t)) p_m(t)
\Bigl(
1-\exp(-x_n(p(t)-p_m(t)))\times
$$
$$
\times\sum\nolimits_{i=1}^r \!\frac{(x_n p(t))^{i-1}}{\Gamma(i)}\Bigm/ \sum\nolimits_{i=1}^r \frac{(x_n p_m(t))^{i-1}}{\Gamma(i)}\Bigr)dt
\to 0
$$
as $n\to\infty$ for any sequence $(x_n)\in {\cal D}(\delta,M)$ under
\begin{equation}
\label{equ:pcv1}
\sup_{x\in I} |p_m(x)-p(x)| = o(\log n) \quad\mbox{as}\quad n\to\infty.
\end{equation}
By (\ref{equ:Holder}), under $m=m(n)$ with $\log^{1/\alpha} n / m \to 0$ as $n\to\infty$ there exists a sequence $\bigl(p_m(x)\bigr)_{m\in\Nb}$, such that (\ref{equ:pcv1}) holds  and the sequence of corresponding maximal $r$-spacings $M_{1,n}^{(r)}$ is satisfying (\ref{equ:mix3}). Thus, under $\log^{1/\alpha} n/m\to 0$,
$$
n\int_{I_m} f_{r}(x_n p_m(t)) p_m(t) dt - n\int_{I} f_{r}(x_n p(t)) p(t) dt \to 0
\quad\mbox{as}\quad n\to\infty,
$$
and $(x_n)\in {\cal D}(\delta,M)$ for any fixed $0\!<\!\delta\!<\! M\!<\!\infty$. Then, by (\ref{equ:mix2}),
$$
n \int_{I_m} f_{n,r}(x_n p_m(t)) p_m(t) dt - n \int_{I} f_{n,r}(x_n p(t)) p(t) dt \to 0\quad\mbox{as}\quad n\to\infty.
$$
Therefore,
\begin{equation}
\label{equ:intcnv}
\sup\nolimits_x \Bigl|
\exp\Bigl(\!- n\! \int_{I_m}\!\! f_{n,r}(x p_m(t)) p_m(t) dt\Bigr)-
\exp\Bigl(\!- n\! \int_{I}\!\! f_{n,r}(x p(t)) p(t) dt\Bigr)
\Bigr|\to 0
\end{equation}
as $n\to\infty$.

Analogously, there exists a sequence of distributions having piecewise constant PDFs $p_m$ and the corresponding CDFs $F_m$ such that the function $G_m(x)=F_m^{-1}(F(x))$ is satisfying the condition $G_m(x_2)-G_m(x_1)\leq x_2-x_1$ for all $x_2>x_1$ and (\ref{equ:intcnv}) is held. Then, for maximal $r$-spacings $M_{1,n}^{-(r)}$,
$$
\P(M_{1,n}^{(r)}<x)\leq \P(M_{1,n}^{-(r)}<x).
$$

Finally, we need to extend Theorem \ref{teo:mix3} to the case  $m=m(n)$, where $\log^{1/\alpha} n/m\to 0$ as $n\to\infty$. We use for simplicity the same notations as in Theorem~\ref{teo:mix3} and Lemma \ref{lem:loc3} equiped by the additional index $m$. Consider a sequence of PDFs $p_m$ of the form (\ref{equ:stden}) such that $c_{\min,m}>\kappa>0$, $d_m<K<\infty$ for all $m\in\Nb$ and $(J_{\tilde i,m})$ is such that $|J_{\tilde i,m}|^{-1}=O(m)$ as $m\to\infty$. Then, there exist sequences $(\delta_m)$, $(s_m)$ such that $s_m\to\infty$, $n\delta_m\to\infty$ as $n\to\infty$ and for any fixed $m$ there exist $y_1,\ldots,y_{s_m}$ such that $V_{y_i}(\delta_m)\subseteq J_{\tilde i,m}$ and $V_{y_i}(\delta_m)\cap
V_{y_j}(\delta_m)=\emptyset$ for all $i,j$, where $i\not= j$. Then, (\ref{equ:c1}) holds with $\delta=\delta_m$ as $m\to\infty$. Taking into account that
$$
\P(M_{1,n}^{(r)}>\gamma_m)\leq \P(M_{U,1,n}^{(r)}>\gamma_m c_{\min,m})\to 0\quad\mbox{as}\quad n\to\infty
$$
under $n\gamma_n\to\infty$,
we conclude that Lemma~\ref{lem:loc3} holds for the sequence $p_m$ under $m=o(n)$ as $n\to\infty$ and such that $c_{\min,m}>\kappa>0$ for all $m$.
Let
$$
\tilde f_{s,r}(x)=-s^{-1}\log \P(s M_{U,1,s}^{(r)}<x),
$$
where $M_{U,1,s}^{(r)}$ is the maximal uniform $r$-spacing, $s\in\mathbb N$. Replacement $f_{n_i,r}$ by $\tilde f_{n_i,r}$
implies the exact equality in (\ref{equ:mix5}) with $o(1)\equiv 0$.
Implication from (\ref{equ:mix5}) to (\ref{equ:mix6}) is valid if \
$$
\max_{i=1,\ldots,m} (1-n_{i,m}/k_{i,m})=o_P(\log^{-1} n)\quad\mbox{as}\quad n\to\infty
$$
and
\begin{equation}
\label{equ:int2}
\max_{i=1,\ldots,m} \sup_{x>0}|\P(M_{U,n_{i,m}}^{(r)}\!\!<\!x)-\P(M_{U,k_{i,m}}^{(r)}\!\!<\!x)|=o(m^{-1})\quad\mbox{as}\;\; m\to\infty,
\end{equation}
where $k_{i,m}=\E\, n_{i,m}=n\theta_{i,m}$. Taking into account that $n_i/n=F_{n,m}(J_{i,m})$ and $\theta_i=F_m(J_{i,m})$, where $F_{n,m}$ is the empirical CDF based on the sample of size $n$ from the distribution $F_m$,
we obtain by Kolmogorov's theorem that
$$
\sup_{m}\max_{i=1,\ldots,m} |n_{i,m}/n-\theta_{i,m}|=O_P(1/\sqrt{n})\quad\mbox{as}\quad n\to\infty.
$$
Therefore, $\max_{i=1,\ldots,m} (1-n_{i,m}/k_{i,m})=O_P(m/\sqrt{n})$ under the conditions $\theta_{i,m}>c/m$ for all $i=1,\ldots,m$, some $c>0$ and sufficiently large $m$.
Remark that
$
\P(\widetilde M_{s_1}^{(r)}<x)\leq \P(\widetilde M_{s_2}^{(r)}<x)+(s_2-s_1)/s_2
$
for any $s_1<s_2$ (see Section~\ref{sec:unif} for the representation of uniform $r$-spacings via exponential random variables). Then,
$$
\max_{i=1,\ldots,m} \sup_{x>0}|\P(M_{U,n_{i,m}}^{(r)}<x)-\P(M_{U,k_{i,m}}^{(r)}<x)|=O\bigl(\sqrt{m/n}\,\bigr)
$$
and (\ref{equ:int2}) holds under $m=o(n^{1/3})$. Now, by (\ref{equ:mix2}) we can replace $\tilde f_{n,r}$ by $f_{n,r}$ in (\ref{equ:mix6}).
Therefore, if to use a sequence $m=m(n)$ so that $c_1\log^{1/\alpha_-} n\leq m \leq c_2 n^{\beta}$
for some $\beta<1/3$, $\alpha_- <\alpha$ and $c_1,c_2>0$ with
$\min_i \theta_{i,m} \geq c / m$ for some $c >0$, the required property (\ref{equ:mix3}) holds.
\end{proof}

We use notations of Theorem \ref{teo:mix1} without the assumption $p_{\min}>0$ and $\widetilde T=\bigcup_{i=1}^s \widetilde I_{i}$ in the following lemma. 

\begin{lemma}
\label{lem:loc5}
Let $X_1,\ldots,X_n$ be a sample from an absolutely continuous distribution having PDF $p$,   $T_{\delta}=\bigcup\nolimits_{j=1}^s I_{j\delta}=\{x\in T \mid p(x)<p_+ -\delta\}$ for some $\delta>0$ and $I_{j\delta}$ be the intervals, $j=1,\ldots,s$, $s\geq q$. Introduce, $\sigma_{k,n}^{(r)}$ are such that $M_{k,n}^{(r)}=X_{\sigma_{k,n}^{(r)},n}-X_{\sigma_{k,n}^{(r)}+r,n}$. Then,
$$
\P([X_{\sigma_{k,n}^{(r)},n},X_{\sigma_{k,n}^{(r)}+r,n}]\subset \widetilde T)\to 1\quad\mbox{as}\quad
n\to\infty.
$$
\end{lemma}

\begin{proof}
Consider $U_i=F(X_i)$, $i=1,\ldots,n$. Then, $U_1,\ldots,U_n$ are {\em i.i.d.} random variables having the standard uniform distribution. Remark that, $Q_{\delta}=F(T_{\delta})$ can be represented as $Q_{\delta}=\bigcup\nolimits_{j=1}^s Q_{j\delta}$, where $Q_{j\delta}=F(I_{j\delta})$ are disjoint intervals, $j=1,\ldots,s$.

Let $J$ be an arbitrary subinterval of $[0,1]$ of length $b>0$.
First we prove that for any interval $J=(x_1,x_2)\subset [0,1]$
\begin{equation}
\label{equ:unis}
\lim_{m\to\infty}\lim_{n\to\infty}
\P\Bigl(\bigcup\nolimits_{k=1}^m [U_{\tilde\sigma_{k,n}^{(r)},n},U_{\tilde\sigma_{k,n}^{(r)}+r,n}]\subseteq J\Bigr)= 1,
\end{equation}
where $\tilde\sigma_{k,n}^{(r)}$ are the such indexes that  
$U_{\tilde\sigma_{k,n}^{(r)}+r,n}-U_{\tilde\sigma_{k,n}^{(r)},n}$ are $k$-th maximal $r$-spacings. For $r=1$ this property follows immediately from symmetry
$$
\P(\tilde\sigma_{1,n}^{(1)}=i_1,\ldots,\tilde\sigma_{m,n}^{(1)}=i_m)=
(n-m+1)!/ (n+1)!
$$
for any arrangement $(i_1,\ldots,i_m)$ of indexes $(0,\ldots,n)$. Remark that
$X_{[n x_1],n}\to x_1$ and $X_{[nx_2],n}\to x_2$ as $n\to\infty$ almost sure. Then, (\ref{equ:unis}) follows immediately from
\begin{eqnarray*}
&&\mspace{-50mu}\lim_{m\to\infty}\lim_{n\to\infty}
\P\Bigl(\bigcup\nolimits_{i=1}^m \tilde\sigma_{k,n}^{(1)}\in  [n_{1\epsilon},n_{2\epsilon}]\Bigr)\\
&& \geq
\lim_{m\to\infty}\lim_{n\to\infty}
\Bigl(1-\prod\nolimits_{k=1}^m (1-(n_2-n_1)/(n-k+1)\Bigr)=1,
\end{eqnarray*}
where $n_{1\epsilon}=[n(x_1+\epsilon)]$ and $n_{2\epsilon}=[n(x_2-\epsilon)]$.
For the case $r>1$ set $\hat\sigma_{1,n}^{(r)}=\tilde\sigma_{1,n}^{(r)}$ and restrict the set of indexes to $G_1=\{i\in 0,\ldots,n-r+1 \mid |i-\hat\sigma_{1,n}^{(r)}|\geq r\}$. Recursively define 
$\hat\sigma_{k,n}^{(r)}$ such that $S_{\hat\sigma_{k,n}^{(r)},n}^{(r)}=\max\{S_{i,n}^{(r)},i\in G_{i-1}\}$ and $G_k=\{i\in G_{k-1} \mid |i-\hat\sigma_{k,n}|\geq r\}$, where $S_{i,n}^{(r)}=U_{i+r,n}-U_i$, $i=0,\ldots,n-r+1$. Then,
$$
\P(\hat\sigma_{1,n}^{(r)}=i_1,\ldots,\hat\sigma_{m,n}^{(r)}=i_m)=
\Bigl( \prod\nolimits_{k=1}^{m} (n-r+2-k(2r-1))\Bigr)^{-1}
$$
for any $i_1,\ldots,i_m$ such that $|i_s-i_r|\geq r$; $s,r\in\{0,\ldots,n-r+1\}$. Therefore, by
$$
\P\Bigl(\bigcup\nolimits_{i=1}^{m(2r-1)} \tilde\sigma_{k,n}^{(r)}\in [n_{1\epsilon},n_{2\epsilon}]\Bigr)\geq
\P\Bigl(\bigcup\nolimits_{i=1}^{m} \hat\sigma_{k,n}^{(r)}\in [n_{1\epsilon},n_{2\epsilon}]\Bigr)
$$
and
$
\lim_{m\to\infty}\lim_{n\to\infty}
\P\Bigl(\bigcup\nolimits_{i=1}^m \hat\sigma_{k,n}^{(r)}\in [n_{1\epsilon},n_{2\epsilon}]\Bigr)=1
$
we obtain (\ref{equ:unis}).

Secondly, remark that
$$
\P\Bigl(c\sum\nolimits_{i=1}^{n-m} E_i/(n-i+1)\geq \sum\nolimits_{i=1}^{n} E_i/(n-i+1)\Bigr)\to 1
$$
as $n\to\infty$ for any fixed $c>1$, where $E_1,E_2,\ldots$ are {\em i.i.d.} the
standard exponential random variables. Using representation for the exponential
order statistics $E_{k,n}=\sum_{i=1}^k E_i/(n-i+1)$ and the exponential
distribution representation for  uniform spacings, we obtain that
$
\P(c M_{U,k,n}^{(r)} \geq M_{U,1,n}^{(r)})\to 1
$
as $n\to\infty$ for any fixed $c>1$ and $k\in\mathbb N$.

Finally, taking into account that
$
(p_+ -\delta) (F^{-1}(x_2)-F^{-1}(x_1))\geq x_2-x_1
$
and
$
p_+(F^{-1}(y_2)-F^{-1}(y_1))\leq y_2-y_1
$
for any $x_1,x_2\in B_{j\delta}$, $x_1<x_2$ and $y_1,y_2\notin B$, $y_1<y_2$,
the first and the second parts of the prove one gets
$$
\P\bigl([X_{\sigma_{k,n}^{(r)},n},X_{\tilde\sigma_{k,n}^{(r)}+r,n}]\cap T_{\delta}=\emptyset\bigr)\to 0\quad
\mbox{as}\quad n\to\infty.
$$
Therefore, $\P\bigl([X_{\sigma_{k,n}^{(r)},n},X_{\tilde\sigma_{k,n}^{(r)}+r,n}]\subset
\widetilde T\bigr)\to 1$ as $n\to\infty$.
\end{proof}

\begin{remark}
\label{rem:2}
\rm
Let $x_0$ be such that $p(x_0)>p_{\min}$ and $T_{\delta}=\{x\mid p(x)<p(x_0)-\delta\}= \bigcup\nolimits_{j=1}^s I_{j\delta}$ for some $\delta>0$, where
$I_{j\delta}$ are finite or infinite disjoint intervals. Then, by Lemma
\ref{lem:loc5}
$$
\P\Bigl(x_0\in [X_{\sigma_{k,n}^{(r)},n},X_{\sigma_{k,n}^{(r)}+r,n}]\Bigr)\to 0\quad
\mbox{as}\quad n\to\infty.\medskip
$$
\end{remark}

\begin{lemma}
\label{lem:loc}
Let $g(x)=n\int_{\widetilde T} f_r(xp(u))p(u)du$, $\bar g(x)=n\int_{\overline T} f_r(xp(u))p(u)du$, where $\overline T=\{x\in [A,B]\mid p(x)>p_+\}$. Then, for any $M>0$ and any sequence $\{x_n\}$: $\limsup_{n\to\infty}  g(x_n)\leq M$,
$$
\lim_{n\to\infty} \bar g(x_n)=0.
$$
\end{lemma}

\begin{proof}
Introduce 
$
g_{\delta}(x)=n\int_{\widetilde T_{\delta}} f_r(xp(u))p(u)du.
$
Then,
$$
g_{\delta}(x_n)=n f_r(x_n \tilde p_n) \tilde p_n \quad \mbox{and} \quad g(x_n)=n f_r(x_n \bar p_n) \bar p_n
$$
for some $\tilde p_n\in [p_{\min},p_+-\delta]$ and $\bar p_n\geq p_+$. Taking into account that $g_{\delta}(x_n)\leq g(x_n)\leq M$ under sufficienlly large $n$ we conclude that $x_n \tilde p_n \geq \log n+o(\log n)$. Then, 
$
x_n \bar p_n \geq \bar p_n / \tilde p_n \log n  +o(\log n)\geq p_+/(p_+-\delta) \log n+o(\log n)
$
and, therefore, $\bar g(x_n)\to 0$ as $n\to\infty$.
\end{proof}

\begin{proof}[Proof of Theorem \ref{teo:mix1}]
Let $\alpha\!=\!\int_{\widetilde T} p(u)du$ and $\overline{M}_{1,k}^{(r)}$ be the maximal $r$-spacing based on a sample of size $k$ from the truncated distribution with the PDF $p^*(x)=\alpha^{-\!1} p(x) \indi_{\{x\in \widetilde T\}}$.
By Lemma \ref{lem:loc5} and the law of large numbers,  
\begin{eqnarray*}
\P\bigl(nM_{1,n}^{(r)}<x)&=&\P\bigl(nM_{1,n}<x\,\bigm|\,[X_{\sigma_n^{(r)}\!\!\!,\,n},X_{\sigma_n^{(r)}\!\!+r,n}]\subset\widetilde T\bigr)+o(1)\\
&=&\P\bigl(n\overline{M}_{1,k}^{(r)}<x\bigr)+o(1),
\end{eqnarray*}
where $M_{1,n}^{(r)}=X_{\sigma_{n}^{(r)}\!\!+r,n}-X_{\sigma_{n}^{(r)}\!\!\!,\,n}$
and $k=[n\alpha]$. By Lemma \ref{lem:mix1} and (\ref{equ:mix3}),
\begin{eqnarray*}
\P(n\overline M_{1,k}^{(r)}<x)&=&\exp\Bigl(-k\int_{\widetilde T} f_{k,r}(\alpha xp^*(u))p^*(u)du\Bigr)+o(1) \\
&=&\exp\Bigl(-n\int_{\widetilde T} f_{r}(xp(u))p(u)du\Bigr)+o(1).
\end{eqnarray*}
By Lemma {\ref{lem:loc}},  
\begin{eqnarray*}
\sup_x\Bigl|\exp\Bigl(\!\!&-&\!\!n\int_{\widetilde T} f_{r}(x p(u))p(u)du\Bigr)-
\exp\Bigl(-n\int_{A}^B f_{r}(x p(u))p(u)du\Bigr) 
\Bigr| \\
&=&\sup_x\bigl[\exp(-\tilde g(x)) (1-\exp(-\bar g(x)))\bigr]\to 0 
\end{eqnarray*}
as $n\to\infty$. The theorem is proved. 
\end{proof}

Now we are going to prove results of Section~\ref{sec:wide}.

\begin{proof}[Proof of Theorem \ref{teo:rspacing}]
We use essentially the same arguments
as Deheuvels \cite{dvs86} to prove (i)-(iii),  with
\begin{equation}
\label{equ:spsr}
S_{n-j,n}^{(r)}=\sum\nolimits_{l=0}^{r-1} S_{n-j+l,n}.
\end{equation}
for $j=1,\ldots,n-r$. More precisely, Deheuvels \cite{dvs86} proved the result for
$r=1$ in two steps. On one hand, for any fixed $N$, he proved the convergence
of
$c_n^{-1}(S_{n-1,n}^{(r)},\ldots,S_{n-N,n}^{(r)})$ ($c_n=a_n$ in case (i),
$c_n=b_n$
in case (ii) and $c_n=(B-b_n)$ in case (iii)) to some specified distribution
$H_{N}$ as $n\to\infty$. On the other hand,  for any fixed $x$ he  proved that
\begin{equation}
\label{equ:spsrn}
\limsup_{n\to\infty} \P\Bigl(\bigcup\nolimits_{j=N+1}^{n-1}\{c_n^{-1} S_{n-j,n}^{(r)}>x\} \Bigr)
\to 0\quad \mbox{as}\quad N\to\infty.
\end{equation}
Then, the distribution of $k$-th maximal $r$-spacing is
the limit distribution of $k$-th maximal $r$-spacing from $H_N$ as $N\to\infty$.
The limit distributions of $c_n^{-1}(S_{n(r)-1,n}^{(r)},\ldots$
$S_{n(r)-N,n}^{(r)})$ for $r>1$ are obtained in a similar way using
(\ref{equ:spsr}). Taking into account that
$$
\P\Bigl(\bigcup\nolimits_{j=N+1}^{n-1}\{c_n^{-1} S_{n-j,n}^{(r)}>x\} \Bigr)\leq  \P\Bigl(\bigcup\nolimits_{j=N+r}^{n-1}\{c_n^{-1} S_{n-j,n}>x/r\} \Bigr)
$$
and (\ref{equ:spsr}), we obtain (\ref{equ:spsrn}) for $r>1$.
\end{proof}

The essential part of Corollary \ref{cor:symmetry} is Lemma \ref{lem:loc5}.

\begin{proof}[Proof of Corollary \ref{cor:symmetry}]
Let $m_F$: $F(m_F)=1/2$ be the median. We lose no generality by the assumption $m_F=0$.
Then, by symmetry the original distribution function can be represented as a mixture
$F(x)=1/2 F_+(x)+1/2 F_-(x)$, where
$$
F_+(x)=1-F_{-}(-x)=(2F(x)-1)\indi_{\{x>0\}}.
$$
Let $M_{1,k,n_1}^{(r)}$ and $M_{2,k,n_2}^{(r)}$ be $k$-th maximal $r$-spacings based on the sample from the distributions having CDFs $F_+$ and $F_{-}$ respectively. It is clear, that $M_{1,k,n_1}^{(r)}$ and $M_{2,k,n-n_1}^{(r)}$ are identically distributed under $n=2n_1$. Moreover, $F_+(0_+)>0$ and, therefore, Theorem \ref{teo:rspacing} is applicable to a sample from the distribution $F_+$. By Lemma \ref{lem:loc5} with Remark \ref{rem:2},
$$
\bigl|\P(M_{k,n}^{(r)}<x|Q_D)-\P(M_{1,k,n_1}^{(r)}<x)\P(M_{2,k,n-n_1}^{(r)}<x)\bigr|\to 0 \quad\mbox{as}\quad n\to\infty
$$
in probability,
where $Q_{D}=\{X_{j}\in [a,b), j\in D\}\cap \{X_{j}\in [b,c],j\notin D\}$ with $\# D=n_1$
Then, taking into account that $n_1/n\to 1/2$ as $n\to\infty$ in probability the result follows immediately by Theorem \ref{teo:rspacing} for $M_{1,k,n_1}^{(r)}$ and $M_{2,k,n-n_1}^{(r)}$.
\end{proof}

\begin{remark}\rm
Lemma \ref{lem:loc5} allows to use mixing of singular
distributions with
different tails, but formulation of general results of such type may be complex.
\end{remark}


\bibliographystyle{elsarticle-num}
\bibliography{base}

\end{document}